\documentclass{amsart}
\usepackage{amssymb}
\usepackage{mathrsfs}
\usepackage{amsmath}
\usepackage{amsfonts}
\usepackage{pifont}
\usepackage{graphicx,amssymb,mathrsfs,amsmath}
\usepackage{tocvsec2}
\usepackage{color}
\newtheorem{thm}{Theorem}[section]
\newtheorem{cor}[thm]{Corollary}
\newtheorem{lem}[thm]{Lemma}
\newtheorem{prop}[thm]{Proposition}
\newtheorem{prob}[thm]{Problem}
\theoremstyle{definition}
\newtheorem{defn}[thm]{Definition}
\newtheorem{example}[thm]{Example}
\theoremstyle{remark}
\newtheorem{rem}[thm]{Remark}
\numberwithin{equation}{section}

\begin{document}
\title[Disjoint Li-Yorke chaos in Fr\' echet spaces]{Disjoint Li-Yorke chaos in Fr\' echet spaces}

\author{Marko Kosti\' c}
\address{Faculty of Technical Sciences,
University of Novi Sad,
Trg D. Obradovi\' ca 6, 21125 Novi Sad, Serbia}
\email{marco.s@verat.net}

{\renewcommand{\thefootnote}{} \footnote{
\\ \text{  }  \ \    {\it Key words and phrases.} disjoint Li-Yorke chaos, disjoint Li-Yorke irregular vectors, disjoint distributional chaos, disjoint hypercyclicity, multivalued linear operators, Fr\' echet spaces.
\\  \text{  }  \ \ The author is partially supported by grant 174024 of Ministry of Science and Technological Development, Republic of Serbia.}}

\begin{abstract}
The main aim of this paper is to consider various notions of (dense) disjoint Li-Yorke chaos for general sequences of multivalued linear operators in Fr\' echet spaces. We also consider continuous analogues of introduced notions and provide certain applications to the abstract partial differential equations.
\end{abstract}
\maketitle

\section{Introduction and Preliminaries}\label{intro}

Assume that $X$ is a Fr\' echet space. As it is well known, a
linear operator $T$ on $X$ is called hypercyclic iff there
exists an element $x\in D_{\infty}(T)\equiv \bigcap_{n\in {\mathbb
N}}D(T^{n})$ whose orbit $\{ T^{n}x : n\in {{\mathbb N}}_{0} \}$ is
dense in $X;$ $T $ is called topologically transitive,
resp. topologically mixing, iff for every pair of open
non-empty subsets $U,\ V$ of $X,$ there exists $n_{0}\in {\mathbb
N}$ such that $T^{n_{0}}(U) \ \cap \ V \neq \emptyset ,$ resp. there
exists $n_{0}\in {\mathbb N}$ such that, for every $n\in {\mathbb
N}$ with $n\geq n_{0},$ $T^{n}(U) \ \cap \ V \neq \emptyset .$ We accept the following notion of chaos: a
linear operator $T$ on $X$ is called chaotic iff it is topologically transitive and the set of periodic points of $T,$ defined by $\{x\in D_{\infty}(T) : (\exists n\in {\mathbb N})\, T^{n}x=x\},$ is dense in $X.$
For further information concerning topological dynamics of linear operators in Banach and Fr\' echet spaces, we refer the reader to the monographs \cite{bayart} by F. Bayart, E. Matheron,
\cite{erdper} by K.-G. Grosse-Erdmann, A. Peris, and a forthcoming one \cite{novascience} by the author.  

The notion of a Li-Yorke irregular vector in Hilbert space has been defined for the first time by B. Beauzamy in \cite{bea-b}. After that, Li-Yorke linear dynamics in Hilbert, Banach and Frechet function spaces has been analyzed by a great number of other authors including  G. T. Pr\v ajitur\v a \cite{parij}, T. Berm\'udez et al \cite{2011}, N. C. Bernardes Jr et al \cite{band}, X. Wu \cite{xwu} and Z. Yin et al \cite{countable} (see also \cite{nis-dragan}, \cite{cheli}).
It is well known that any linear hypercyclic operator needs to be Li-Yorke chaotic as well as that the converse statement does not hold in general. On the other hand, the notion of distributional chaos was introduced by B. Schweizer and J. Sm\' ital in \cite{smital} (1994). In linear dynamics, 
distributional chaos was firstly considered in the analyses of quantum harmonic oscillator, by J. Duan et al \cite{duan} (1999) and P. Oprocha \cite{p-oprocha} (2006); the first systematic studies of linear distributional chaos is those ones carried out by 
N. C. Bernardes Jr. et al \cite{2013JFA} (2013) and J. A. Conejero et al \cite{mendoza} (2016). Further information about Li-Yorke chaos and distributional chaos in metric and Fr\' echet spaces can be obtained by consulting \cite{novascience} and references cited therein.

Disjointness in linear dynamics was introduced independently by L. Bernal--Gonz\'alez \cite{bg07} (2007) and J. B\`es, A. Peris \cite{bp07} (2007). From then on, a great number of various notions of disjointness for linear operators has been introduced and analyzed; for the notion of disjoint mixing operators and disjoint supercyclic operators one may refer e.g. to the article \cite{bm201345}
by J. B\`es et al and the doctoral dissertation \cite{ma10} of \"O. Martin, respectively. Regarding disjoint dynamics of abstract partial differential equations, the first step has been made by the author
in \cite[Subsection 3.1.1]{knjigaho}, where disjointness for strongly continuous semigroups induced by semiflows has been examined. These results have been
recently reconsidered in a joint research study \cite{ckpv} with 
C.-C. Chen, S. Pilipovi\'c and D. Velinov for $C$-distribution semigroups and $C$-distribution cosine functions in Fr\' echet spaces, as well as in a joint research study with V. E. Fedorov \cite{peruvian-waq} for abstract degenerate fractional differential equations in Fr\'echet spaces.
The notion of disjoint (reiterative, ($m_{n}$-)) distributionally chaotic operators and some applications to abstract PDEs have been investigated in \cite{nsjom-novi}-\cite{ddcNS} and \cite{cheli}.
To the best knowledge of the author, this is the first paper, in both linear and non-linear setting, which considers the notion of disjoint Li-Yorke chaos. The genesis of paper is motivated by our recent results on the existence of special types of dense Li-Yorke irregular manifolds obtained in a joint research study with A. Bonilla \cite{bk}, and later expanded by the author in \cite{reit}. 

The organization and main ideas of paper are given as follows. First of all, we recollect some necessary preliminaries about lower and upper $m_{n}$-densities of subsets in ${\mathbb N};$ after that, in Subsection \ref{MLOs}, we remind ourselves of the basic facts and definitions from the theory of multivalued linear operators in Fr\' echet spaces.
Various notions of disjoint Li-Yorke chaos were introduced and analyzed in Section \ref{marek-prom}: In Definition \ref{DC-unbounded-fric-DISJOINT-emenly}, we introduce the notion of $(d,\tilde{X},m_{n},s,i)$-Li-Yorke chaos, where $1\leq s,i\leq 2;$ in Definition \ref{DC-unbounded-fric-DISJOINT-emen-ly}, we introduce the notion of $(d,\tilde{X},s,i)$-Li-Yorke chaos, where $1\leq s\leq 2$
and $3\leq i\leq 4,$ and finally, in Definition \ref{fabrika}, we introduce the notion of $(d,\tilde{X},m_{n},s,i)$-Li-Yorke chaos, where $3\leq s\leq 4$ and $1\leq i\leq 2.$ Here, $(m_{n})$ denotes an increasing sequence of positive reals satisfying $\liminf_{n\rightarrow \infty}\frac{m_{n}}{n}>0.$
Disjoint Li-Yorke chaos under our consideration generalizes the notions of disjoint hypercyclicity and disjoint distributional chaoticity for multivalued linear operators, and it cannot be reduced to the Li-Yorke chaos of single components, as illustrated in Example \ref{count-grof-ly}. In Proposition \ref{beqa}, we clarify the fundamental inclusions for various types of disjoint Li-Yorke chaos. Following a simple observation from \cite{bk}, in Proposition \ref{banach-space}, we connect a certain type of disjointness condition of Li-Yorke type with reiterative distributional chaos of type $0.$
Disjoint Li-Yorke irregular vectors and manifolds have been investigated in Subsection \ref{emojstevo}. 

Speaking-matter-of-factly, the starting point for genesis of this paper is the following theorem, which can be deduced obeying the method developed in the proof of \cite[Theorem 15]{2013JFA} and our previous considerations of dense (disjoint, reiterative) $m_{n}$-distributional chaos contained in the papers \cite{bk} and \cite{ddc}-\cite{ddcNS}; from the sake of completeness, we will outline the main details of proof  in Section \ref{main-reza}:

\begin{thm}\label{na-dobro1-ly}
Suppose that $X$ is separable, $m\in {\mathbb N},$ $((T_{j,k})_{k\in {\mathbb N}})_{1\leq j\leq N}$ is a sequence in $L(X,Y),$
$X_{0}$ is a dense linear subspace of $X,$ $(m_{n})\in {\mathrm R}$ as well as:
\begin{itemize}
\item[(i)] $\lim_{k\rightarrow \infty}T_{j,k}x=0,$ $x\in X_{0},$ $j\in {\mathbb N}_{N};$
\item[(ii)] there exist a vector $y\in X$ and an increasing sequence $(n_{k})$ tending to infinity such that $\lim_{k\rightarrow \infty}p_{m}^{Y}(T_{j,n_{k}}y)=+\infty ,$ $j\in {\mathbb N}_{N}$ [$\lim_{k\rightarrow \infty}\|T_{j,n_{k}}y\|_{Y}=+\infty ,$ $j\in {\mathbb N}_{N}$, provided that $Y$ is a Banach space].
\end{itemize}
Then there exists a dense
submanifold $W$ of $X$ consisting of those vectors $x$ which are $(d,m_{n})$-distributionally near to zero of type $1$ for $((T_{j,k})_{k\in {\mathbb N}})_{1\leq j\leq N}$ and for which there exists a strictly increasing subsequence $(l_{k})$ of $(n_{k})$ such that
the sequence $(p_{m}(T_{j,l_{k}}x))_{k\in {\mathbb N}}$ tends to $+\infty$ for all $j\in {\mathbb N}_{N}$ [$(\|T_{j,l_{k}}x\|_{Y})_{k\in {\mathbb N}}$ tends to $+\infty$ for all $j\in {\mathbb N}_{N},$ provided that $Y$ is a Banach space]. In particular, $((T_{j,k})_{k\in {\mathbb N}})_{1\leq j\leq N}$ is densely $(d,W,m_{n},1,1)$-Li-Yorke chaotic.
\end{thm}

In a separate part of the third section, Subsection \ref{shiftojed}, we analyze certain corollaries of Theorem \ref{na-dobro1-ly} and applications to unilateral backward weighted shift operators and weigted forward shift operators, which need not be continuous in our investigation (see also Example \ref{count-grof-ly324}).
The Li-Yorke chaos for translation semigroups in weighted function spaces have been considered for the first time
by X. Wu in \cite{xwu}, who proved that a strongly continuous translation semigroup $(T(t))_{t\geq 0}$ is Li-Yorke chaotic on $
L^{p}_{\rho}([0,\infty))$ iff $\liminf_{t\rightarrow +\infty}\rho(t)=0,$ 
provided in advance that the function $\rho$ is bounded from above; in this case, being Li-Yorke chaotic and hypercyclic for $(T(t))_{t\geq 0}$ is the same thing.
This is no longer case if the function $\rho$ is not bounded from above, when there exists a strongly continuous translation semigroup that is Li-Yorke chaotic, even completely distributionally chaotic, but not hypercyclic (\cite{bara}).
It is worth noting that X. Wu has analyzed in \cite{xwu} several various notions similar to Li-Yorke chaos, like sensitivity, infinite sensitivity, spatio-temporal chaos, dense $\delta$-chaos and generic ($\delta$-)chaos; see \cite{band} for discrete analogues. The consideration
of these notions for disjointness requires further analyses and it is without scope of this paper; at this place, we want only to mention in passing that the assertion of \cite[Lemma 2.1]{xwu} holds not only for a translation $C_{0}$-semigroup on a weighted function space but also for any strongly continuous operator family $(T(t))_{t\geq 0}\subseteq L(X,Y),$ where $X$ and $Y$ are possibly different Fr\' echet spaces. 
The analysis from \cite{xwu} has been continued by the author in \cite{nis-dragan}, especially for translation semigroups and semigroups induced by semiflows in weighted function spaces. 

Disjoint Li-Yorke chaotic properties of abstract PDEs in Fr\'echet spaces have been analyzed in Section \ref{well-posed}, where we initiate the studies of (disjoint) Li-Yorke chaos for abstract differential equations of second order and (disjoint) Li-Yorke chaos for abstract differential equations of fractional order in time variable; albeit formulated for disjoint $(f,1,1)$-Li-Yorke chaos, a great number of examples are given for disjoint $(f,1,1)$-distributional chaos which is a much stronger notion, whose discrete counterpart has been recently analyzed in \cite{ddcNS} (here, $f : [0,\infty) \rightarrow [1,\infty)$ is an increasing mapping satisfying $\liminf_{t\rightarrow +\infty}\frac{f(t)}{t}>0$). This is a foundational study of disjoint Li-Yorke chaos and we would like to say that we have found this theme very difficult to be thoroughly explored in theoretical sense, primarily from the facts that the methods from \cite{2011}, \cite{band}, \cite{nis-dragan} and \cite{parij} cannot be so easily reexamined for disjointness. For the sake of brevity, we propose only one open problem in Subsection \ref{emojstevo} (Problem \ref{pizdatistrina}).

We use the standard terminology throughout the paper. We assume that $X$ and $Y$ are two
non-trivial Fr\' echet spaces over the same field of scalars ${\mathbb K}\in \{ {\mathbb R}, {\mathbb C} \}$ as well as that the topologies of $X$ and $Y$ are 
induced by the fundamental systems $(p_{n})_{n\in {\mathbb N}}$ and $(p_{n}^{Y})_{n\in {\mathbb N}}$ of
increasing seminorms, respectively (separability of $X$ or $Y$ is not assumed a priori in future). Then the translation invariant metric\index{ translation invariant metric} $d :
X\times X \rightarrow [0,\infty),$ defined by
\begin{equation}\label{metri}
d(x,y):=\sum
\limits_{n=1}^{\infty}\frac{1}{2^{n}}\frac{p_{n}(x-y)}{1+p_{n}(x-y)},\
x,\ y\in X,
\end{equation}
enjoys the following properties:
$
d(x+u,y+v)\leq d(x,y)+d(u,v),$ $x,\ y,\ u,\
v\in X;
$
$d(cx,cy)\leq (|c|+1)d(x,y),$ $ c\in {\mathbb K},\ x,\ y\in X,
$
and
$
d(\alpha x,\beta x)\geq \frac{|\alpha-\beta|}{1+|\alpha-\beta|}d(0,x),$ $x\in X,$ $ \alpha,\ \beta \in
{\mathbb K}.$
Define the translation invariant metric $d_{Y} :
Y\times Y \rightarrow [0,\infty)$ by replacing $p_{n}(\cdot)$ with
$p_{n}^{Y}(\cdot)$ in (\ref{metri}). We endow the Fr\' echet space $Y^{k}$ with the metric 
$$
d_{Y^{k}}(\vec{x},\vec{y}):=\max_{1\leq j\leq k}d_{Y}(x_{j},y_{j}),\quad \vec{x}=(x_{1},\cdot \cdot \cdot,x_{k})\in Y^{k},\ \vec{y}=(y_{1},\cdot \cdot \cdot,y_{k})\in Y^{k},
$$
where $k\in {\mathbb N}.$

Suppose that $C\in L(X)$ is an injective operator.
Put $p_{n}^{C}(x):=p_{n}(C^{-1}x),$ $n\in {\mathbb N},$ $x\in R(C).$ Then
$p_{n}^{C}(\cdot)$ is a seminorm on $R(C)$ and the calibration
$(p_{n}^{C})_{n\in {\mathbb N}}$ induces a Fr\' echet locally convex topology on
$R(C);$ we shall denote this space simply by $[R(C)].$ Notice
that $[R(C)]$ is a Fr\' echet (Banach) space provided that
$X$ is. 

Given $s\in {\mathbb R}$ in advance, set $\lfloor s \rfloor:=\sup \{
l\in {\mathbb Z} : s\geq l \}$ and $\lceil s \rceil:=\inf \{ l\in
{\mathbb Z} : s\leq l \}.$ Denote by $E_{\alpha,\beta}(z)$ the Mittag-Leffler
function $E_{\alpha,\beta}(z):=\sum
 _{n=0}^{\infty}z^{n}/\Gamma(\alpha n+\beta),$ $z\in
{\mathbb C}$. Set, for short, $E_{\alpha}(z):=E_{\alpha, 1}(z),$ $z\in {\mathbb C}$ and $\Sigma_{\vartheta}:=\{re^{i\theta} : |\theta|<\vartheta\}$ ($\vartheta \in (0,\pi]$).
We refer the reader to \cite{knjigaho} for the notions of fractionally integrated $C$-semigroups and $C$-cosine functions, $C$-distribution semigroups and $C$-distribution cosine functions, $\alpha$-times $C$-regularized resolvent families and their integral generators ($\alpha>0$, $C\in L(X)$ injective). Throughout the paper, we assume that $N\in {\mathbb N} \setminus \{1\};$ set ${\mathbb N}_{N}:=\{1,2,\cdot \cdot \cdot, N\}.$

We will use the following notions of lower and upper densities for a subset $A\subseteq {\mathbb N}:$ 

\begin{defn}\label{prckojed} (\cite{F-operatori})
Let $(m_{n})$ be an increasing sequence in $[1,\infty).$ Then:
\begin{itemize}
\item[(i)] The lower $(m_{n})$-density of $A,$ denoted by $\underline{d}_{m_{n}}(A),$ is defined through
$$
\underline{d}_{{m_{n}}} (A):=\liminf_{n\rightarrow \infty}\frac{|A \cap [1,m_{n}]|}{n};
$$
\item[(ii)] The lower $l;(m_{n})$-Banach density of $A,$ denoted by $\underline{Bd}_{l;{m_{n}}}(A),$
is defined through
$$
\underline{Bd}_{l;{m_{n}}}(A):=\liminf_{s\rightarrow +\infty}\liminf_{n\rightarrow \infty}\frac{|A \cap [n+1,n+m_{s}]|}{s}.
$$ 
\end{itemize}
\end{defn}

Denote by ${\mathrm R}$ the class consisting of all increasing sequences $(m_{n})$ of positive reals satisfying $\liminf_{n\rightarrow \infty}\frac{m_{n}}{n}>0,$ i.e., there exists a finite constant $L>0$ such that $n\leq L m_{n},$ $n\in {\mathbb N}.$ Unless stated otherwise, we will always assume that $(m_{n})\in {\mathrm R}$
henceforth. The assumption $m_{n} \in {\mathbb N}$ for all $n\in {\mathbb N}$ can be made.

\subsection{Multivalued linear operators}\label{MLOs}

A multivalued map (multimap) ${\mathcal A} : X \rightarrow P(Y)$ is said to be a multivalued
linear operator (MLO) iff the following two conditions hold:
\begin{itemize}
\item[(i)] $D({\mathcal A}) := \{x \in X : {\mathcal A}x \neq \emptyset\}$ is a linear subspace of $X$;
\item[(ii)] ${\mathcal A}x +{\mathcal A}y \subseteq {\mathcal A}(x + y),$ $x,\ y \in D({\mathcal A})$
and $\lambda {\mathcal A}x \subseteq {\mathcal A}(\lambda x),$ $\lambda \in {\mathbb K},$ $x \in D({\mathcal A}).$
\end{itemize}
If 
$x,\ y\in D({\mathcal A})$ and $\lambda,\ \eta \in {\mathbb K}$ with $|\lambda| + |\eta| \neq 0,$ then it is well-known that 
$\lambda {\mathcal A}x + \eta {\mathcal A}y = {\mathcal A}(\lambda x + \eta y);$ furthermore, if ${\mathcal A}$ is an MLO, then ${\mathcal A}0$ is a linear manifold in $Y$
and ${\mathcal A}x = f + {\mathcal A}0$ for any $x \in D({\mathcal A})$ and $f \in {\mathcal A}x.$ Set $R({\mathcal A}):=\{{\mathcal A}x :  x\in D({\mathcal A})\}.$
The set ${\mathcal A}^{-1}0 = \{x \in D({\mathcal A}) : 0 \in {\mathcal A}x\}$ is called the kernel\index{multivalued linear operator!kernel}
of ${\mathcal A}$ and it is denoted henceforth by $N({\mathcal A})$ or Kern$({\mathcal A}).$ The inverse ${\mathcal A}^{-1}$ of an MLO is defined by
$D({\mathcal A}^{-1}) := R({\mathcal A})$ and ${\mathcal A}^{-1} y := \{x \in D({\mathcal A}) : y \in {\mathcal A}x\}$.
Let us recall that ${\mathcal A}$ is called purely multivalued iff ${\mathcal A}0\neq \{0\}.$ \index{multivalued linear operator!sum}

Suppose that ${\mathcal A} : X \rightarrow P(Y)$ and ${\mathcal B} : Y\rightarrow P(Z)$ are two MLOs, where $Z$ is a Fr\' echet space over the same field of scalars as $X$ and $Y$. The product of ${\mathcal A}$
and ${\mathcal B}$ is defined by $D({\mathcal B}{\mathcal A}) :=\{x \in D({\mathcal A}) : D({\mathcal B})\cap {\mathcal A}x \neq \emptyset\}$ and\index{multivalued linear operator!product}
${\mathcal B}{\mathcal A}x:=
{\mathcal B}(D({\mathcal B})\cap {\mathcal A}x).$ Then ${\mathcal B}{\mathcal A} : X\rightarrow P(Z)$ is an MLO and
$({\mathcal B}{\mathcal A})^{-1} = {\mathcal A}^{-1}{\mathcal B}^{-1}.$ The multiplications of MLOs with scalars and sums of MLOs are taken in the usual way. 
The integer powers of an MLO ${\mathcal A} :  X\rightarrow P(X)$ are defined inductively by: ${\mathcal A}^{0}=:I;$
$$
D({\mathcal A}^{n}) := \bigl\{x \in  D({\mathcal A}^{n-1}) : D({\mathcal A}) \cap {\mathcal A}^{n-1}x \neq \emptyset \bigr\},
$$
and
$$
{\mathcal A}^{n}x := \bigl({\mathcal A}{\mathcal A}^{n-1}\bigr)x =\bigcup_{y\in  D({\mathcal A}) \cap {\mathcal A}^{n-1}x}{\mathcal A}y,\quad x\in D( {\mathcal A}^{n}).
$$
Set $D_{\infty}({\mathcal A}):=\bigcap_{n\in {\mathbb N}}D({\mathcal A}^{n}).$

Suppose that ${\mathcal A}$ is an MLO in $ X.$ Then a point $\lambda \in {\mathbb C}$ is said to be an eigenvalue of ${\mathcal A}$
iff there exists a vector $x\in X\setminus \{0\}$ such that $\lambda x\in {\mathcal A}x;$ we call $x$ an eigenvector of operator ${\mathcal A}$ corresponding to the eigenvalue $\lambda.$ The point spectrum of ${\mathcal A},$ $\sigma_{p}({\mathcal A})$ for short, is defined as the set consisting of all eigenvalues of ${\mathcal A}.$

We need the following definition from \cite{cheli}:

\begin{defn}\label{obori}
We say that the sequence
$({\mathcal A}_{j})_{j\in {\mathbb N}}$ of MLOs is  $\tilde{X}$-Li-Yorke chaotic
iff there exists an uncountable set $S\subseteq \bigcap_{j\in {\mathbb N}}D({\mathcal A}_{j}) \bigcap \tilde{X}$ such that
for every pair $(x,y) \in S \times S$ of distinct points and for every integer $j\in {\mathbb N}$ there exist elements $x_{j}\in {\mathcal A}_{j}x$ and $y_{j}\in {\mathcal A}_{j}y$ so that
\[
 \liminf_{j \to \infty} d_{Y}\bigl(x_{j} ,y_{j} \bigr) = 0
 \ \ \mbox{ and } \ \
 \limsup_{j \to \infty} d_{Y}\bigl(x_{j} ,y_{j}\bigr)> 0.
\]
In this case, $S$ is called a  $\tilde{X}$-Li-Yorke scrambled set for $({\mathcal A}_{j})_{j\in {\mathbb N}}$ and each
such pair $(x,y)$ is called a  $\tilde{X}$-Li-Yorke pair for $({\mathcal A}_{j})_{j\in {\mathbb N}}$. We say that $({\mathcal A}_{j})_{j\in {\mathbb N}}$ is densely $\tilde{X}$-Li-Yorke chaotic iff $S$ can be chosen to be dense in $\tilde{X}.$
\end{defn}

We refer the reader to \cite{cheli}, \cite{band} and \cite{nis-dragan} for the notion and properties of Li-Yorke (semi-)irregular vectors.
Any notion introduced above is accepted also for 
an MLO operator ${\mathcal A} : X \rightarrow P(X)$ by using the sequence $({\mathcal A}_{j}\equiv
{\mathcal A}^{j})_{j\in {\mathbb N}}$ for definition.  Finally, if $\tilde{X}=X,$ then we remove the prefix ``$\tilde{X}$-'' from the terms and notions.

We will also use the following definition from \cite{ddcNS}:

\begin{defn}\label{temica} 
Let $(m_{n})\in {\mathrm R}.$
Suppose that, for every $j\in {\mathbb N}_{N}$ and $k\in {\mathbb N},$ ${\mathcal A}_{j,k} : D({\mathcal A}_{j,k})\subseteq X \rightarrow Y$ is an MLO and $x\in \bigcap_{j=1}^{N}\bigcap_{k=1}^{\infty}D({\mathcal A}_{j,k}),$ $x\neq 0.$ Then we say that $x$ is (reiteratively) $(d,m_{n})$-distributionally
near to $0$ of type $1$ for $(({\mathcal A}_{j,k})_{k\in {\mathbb N}})_{1\leq j\leq N}$
iff there exists $A\subseteq {\mathbb N}$ such that ($\underline{Bd}_{l;m_{n}}(A^{c})=0$)
$\underline{d}_{m_{n}}(A^{c})=0$ as well as for each $j\in {\mathbb N}_{N}$ and $k\in {\mathbb N}$ there exists $x_{j,k}\in {\mathcal A}_{j,k}x$ such that
$\lim_{k\in A,k\rightarrow \infty}x_{j,k}=0,$ $j\in {\mathbb N}_{N}.$ 
\end{defn}

Definition of $(d,\tilde{X},i)$-distributional chaoticity of tuple $(({\mathcal A}_{j,k})_{k\in {\mathbb N}})_{1\leq j\leq N},$ where $i\in {\mathbb N}_{12},$ is too large to be repeated here. For further information on the subject, we refer the reader to \cite{ddc}. 

In this paper, 
we will consider the spaces $L^{p}_{\rho_{1}}(\Omega)$ and $C_{0,\rho}(\Omega),$
where $\Omega$ is an open non-empty subset of ${{\mathbb
R}^{n}}.$ Here, $\rho_{1} : \Omega \rightarrow (0,\infty)$ is a locally
integrable function, the norm of an element $f\in L^{p}_{\rho_{1}}(\Omega)$ is given by $||f||_{p}:=(\int
_{\Omega}|f(x)|^{p} \rho_{1}(x)\, dx)^{1/p}$ and $dx$ denotes Lebesgue's measure on 
${{\mathbb
R}^{n}}.$
Recall that, for a given upper
semicontinuous function $\rho : \Omega \rightarrow (0,\infty)$, the space $C_{0,\rho}(\Omega)$ consists of all continuous functions $f : \Omega \rightarrow
{\mathbb C}$ satisfying that, for every $\epsilon>0,$ $\{x\in \Omega :
|f(x)|\rho(x)\geq \epsilon \}$ is a compact subset of $\Omega;$ equipped
with the norm $||f||:=\sup_{x\in \Omega}|f(x)|\rho(x),$
$C_{0,\rho}(\Omega)$ becomes a Banach space. 

\section{Disjoint Li-Yorke chaos, disjoint Li-Yorke irregular vectors and manifolds}\label{marek-prom}

Let $ \epsilon>0,$ and let $(x_{j,k})_{k\in {\mathbb N}}$ and $(y_{j,k})_{k\in {\mathbb N}}$ be sequences in $X$ ($1\leq j\leq N$). Consider the following conditions:
\begin{align}\label{frechet-banach1}
\begin{split}
& (\exists m\in {\mathbb N})(\forall k\in {\mathbb N})(\exists l_{k}\in {\mathbb N}) \mbox{  s.t. }l_{k}<l_{k+1},\ k\in {\mathbb N},
\\ &\mbox{ and }\lim_{k\rightarrow \infty}p_{m}^{Y}\bigl(x_{j,l_{k}}-y_{j,l_{k}}\bigr)=+\infty, \ k\in {\mathbb N},\ j\in {\mathbb N}_{N},\mbox{ provided that }Y\mbox{ is a Fr\' echet space}, or
\\
& (\forall k\in {\mathbb N})(\exists l_{k}\in {\mathbb N}) \mbox{  s.t. }l_{k}<l_{k+1},\ k\in {\mathbb N},
\\ & \mbox{ and }\lim_{k\rightarrow \infty}\bigl\|x_{j,l_{k}}-y_{j,l_{k}}\bigr\|_{Y}=+\infty, \  j\in {\mathbb N}_{N},\mbox{ provided that }Y\mbox{ is a Banach space};
\end{split}
\end{align}
\begin{align}\label{frechet-banach2}
\begin{split}
& (\exists m\in {\mathbb N})(\forall k\in {\mathbb N})(\forall j\in {\mathbb N}_{N})(\exists l_{k}^{j}\in {\mathbb N}) \mbox{  s.t. }l_{k}^{j}<l_{k+1}^{j},\ k\in {\mathbb N},\ j\in {\mathbb N}_{N},
\\ &\mbox{ and }\lim_{k\rightarrow \infty}p_{m}^{Y}\bigl(x_{j,l_{k}^{j}}-y_{j,l_{k}^{j}}\bigr)=+\infty, \ k\in {\mathbb N},\ j\in {\mathbb N}_{N},\mbox{ provided that }Y\mbox{ is a Fr\' echet space}, or
\\
& (\forall k\in {\mathbb N})(\forall j\in {\mathbb N}_{N})(\exists l_{k}^{j}\in {\mathbb N}) \mbox{  s.t. }l_{k}^{j}<l_{k+1}^{j},\ k\in {\mathbb N},\ j\in {\mathbb N}_{N},
\\ & \mbox{ and }\lim_{k\rightarrow \infty}\bigl\|x_{j,l_{k}^{j}}-y_{j,l_{k}^{j}}\bigr\|_{Y}=+\infty, \  j\in {\mathbb N}_{N},\mbox{ provided that }Y\mbox{ is a Banach space};
\end{split}
\end{align}
\begin{align}\label{bunda1}
\underline{d}_{m_{n}}\Biggl( \bigcup_{j\in {\mathbb N}_{N}} \bigl\{k \in {\mathbb N} : d_{Y}\bigl(x_{j,k},y_{j,k}\bigr)
\geq \epsilon \bigr\}\Biggr)=0;
\end{align}
\begin{align}\label{bunda2}
\underline{Bd}_{l;m_{n}}\Biggl( \bigcup_{j\in {\mathbb N}_{N}} \bigl\{k \in {\mathbb N} : d_{Y}\bigl(x_{j,k},y_{j,k}\bigr)
\geq \epsilon \bigr\}\Biggr)=0;
\end{align}
\begin{align}\label{bunda}
(\forall k\in {\mathbb N})(\exists n_{k}\in {\mathbb N}) \mbox{  s.t. }n_{k}<n_{k+1},\ k\in {\mathbb N}\mbox{ and }\lim_{k\rightarrow \infty}d_{Y}\bigl(x_{j,n_{k}},y_{j,n_{k}}\bigr)=0, \  j\in {\mathbb N}_{N};
\end{align}
\begin{align}\label{bundaw}
(\forall k\in {\mathbb N})(\forall j\in {\mathbb N}_{N})(\exists n_{k}^{j}\in {\mathbb N}) \mbox{  s.t. }n_{k}^{j}<n_{k+1}^{j},\ k\in {\mathbb N}\mbox{ and }\lim_{k\rightarrow \infty}d_{Y}\bigl(x_{j,n_{k}^{j}},y_{j,n_{k}^{j}}\bigr)=0, \  j\in {\mathbb N}_{N}.
\end{align}

Now we are ready to propose the following definitions:

\begin{defn}\label{DC-unbounded-fric-DISJOINT-emenly}
Let $i\in {\mathbb N}_{2}$ and $(m_{n})\in {\mathrm R}.$
Suppose that, for every $j\in {\mathbb N}_{N}$ and $k\in {\mathbb N},$ ${\mathcal A}_{j,k} : D({\mathcal A}_{j,k})\subseteq X \rightarrow Y$ is an MLO and $\tilde{X}$ is a linear
subspace of $X.$
Then we say that the sequence $(({\mathcal A}_{j,k})_{k\in
{\mathbb N}})_{1\leq j\leq N}$ is disjoint
$(\tilde{X},m_{n},1,i)$-Li-Yorke chaotic, $(d,\tilde{X},m_{n},1,i)$-Li-Yorke chaotic in short, resp. disjoint
$(\tilde{X},m_{n},2,i)$-Li-Yorke chaotic, $(d,\tilde{X},m_{n},2,i)$-Li-Yorke chaotic in short,  iff there exists an uncountable
set $S\subseteq \bigcap_{j=1}^{N} \bigcap_{k=1}^{\infty} D({\mathcal A}_{j,k}) \cap \tilde{X}$ such that for each $\epsilon>0$ and for each pair $x,\
y\in S$ of distinct points we have that for each $j\in {\mathbb N}_{N}$ and $k\in {\mathbb N}$ there exist elements $x_{j,k}\in {\mathcal A}_{j,k}x$ and $y_{j,k}\in {\mathcal A}_{j,k}y$ such that
\eqref{frechet-banach1} and  (2.i+2) hold, resp. \eqref{frechet-banach2} and  (2.i+2) hold.

Let $s\in {\mathbb N}_{2}.$ Then we say that the sequence $(({\mathcal A}_{j,k})_{k\in
{\mathbb N}})_{1\leq j\leq N}$ is densely
$(d,\tilde{X},m_{n},s,i)$-Li-Yorke chaotic iff $S$ can be chosen to be dense in $\tilde{X}.$
A finite sequence $({\mathcal A}_{j})_{1\leq j\leq N}$ of MLOs on $X$ is said to be (densely)
$(d,\tilde{X},m_{n},s,i)$-distributionally chaotic\index{$(d,\tilde{X},m_{n},s,i)$-Li-Yorke chaotic operator} iff the sequence $(({\mathcal A}_{j,k}\equiv
{\mathcal A}_{j}^{k})_{k\in {\mathbb N}})_{1\leq j\leq N}$ is.
The set $S$ is said to be $(d,\sigma_{\tilde{X}},m_{n},s,i)$-Li-Yorke scrambled set\index{ $(d,\sigma_{\tilde{X}},m_{n},s,i)$-Li-Yorke scrambled set} ($(d,\sigma,m_{n},s,i)$-Li-Yorke scrambled set in the case\index{(d,$\sigma,m_{n},s,i)$-Li-Yorke scrambled set}
that $\tilde{X}=X$) of $(({\mathcal A}_{j,k})_{k\in
{\mathbb N}})_{1\leq j\leq N}$ ($({\mathcal A}_{j})_{1\leq j\leq N}$);  in the case that
$\tilde{X}=X,$ then we also say that the sequence $(({\mathcal A}_{j,k})_{k\in
{\mathbb N}})_{1\leq j\leq N}$ ($({\mathcal A}_{j})_{1\leq j\leq N}$) is disjoint $(m_{n},s,i)$-Li-Yorke chaotic, $(d,m_{n},s,i)$-Li-Yorke chaotic in short.
\end{defn}

\begin{defn}\label{DC-unbounded-fric-DISJOINT-emen-ly}
Let $i\in \{3,4\}.$
Suppose that, for every $j\in {\mathbb N}_{N}$ and $k\in {\mathbb N},$ ${\mathcal A}_{j,k} : D({\mathcal A}_{j,k})\subseteq X \rightarrow Y$ is an MLO and $\tilde{X}$ is a linear
subspace of $X.$
Then we say that the sequence $(({\mathcal A}_{j,k})_{k\in
{\mathbb N}})_{1\leq j\leq N}$ is disjoint
$(\tilde{X},1,i)$-Li-Yorke chaotic, $(d,\tilde{X},1,i)$-Li-Yorke chaotic in short, resp. disjoint
$(\tilde{X},2,i)$-Li-Yorke chaotic, $(d,\tilde{X},2,i)$-Li-Yorke chaotic in short, iff there exists an uncountable
set $S\subseteq \bigcap_{j=1}^{N} \bigcap_{k=1}^{\infty} D({\mathcal A}_{j,k}) \cap \tilde{X}$ such that for each $\epsilon>0$ and for each pair $x,\
y\in S$ of distinct points we have that for each $j\in {\mathbb N}_{N}$ and $k\in {\mathbb N}$ there exist elements $x_{j,k}\in {\mathcal A}_{j,k}x$ and $y_{j,k}\in {\mathcal A}_{j,k}y$ such that
\eqref{frechet-banach1} and  (2.i+2) hold, resp. \eqref{frechet-banach2} and  (2.i+2) hold.

Let $s\in {\mathbb N}_{2}.$ Then we say that the sequence $(({\mathcal A}_{j,k})_{k\in
{\mathbb N}})_{1\leq j\leq N}$ is densely
$(d,\tilde{X},s,i)$-Li-Yorke chaotic iff $S$ can be chosen to be dense in $\tilde{X}.$
A finite sequence $({\mathcal A}_{j})_{1\leq j\leq N}$ of MLOs on $X$ is said to be (densely)
$(d,\tilde{X},s,i)$-distributionally chaotic\index{$(d,\tilde{X},s,i)$-Li-Yorke chaotic operator} iff the sequence $(({\mathcal A}_{j,k}\equiv
{\mathcal A}_{j}^{k})_{k\in {\mathbb N}})_{1\leq j\leq N}$ is.
The set $S$ is said to be $(d,\sigma_{\tilde{X}},s,i)$-Li-Yorke scrambled set\index{ $(d,\sigma_{\tilde{X}},s,i)$-Li-Yorke scrambled set} ($(d,\sigma,s,i)$-Li-Yorke scrambled set in the case\index{(d,$\sigma,s,i)$-Li-Yorke scrambled set}
that $\tilde{X}=X$) of $(({\mathcal A}_{j,k})_{k\in
{\mathbb N}})_{1\leq j\leq N}$ ($({\mathcal A}_{j})_{1\leq j\leq N}$);  in the case that
$\tilde{X}=X,$ then we also say that the sequence $(({\mathcal A}_{j,k})_{k\in
{\mathbb N}})_{1\leq j\leq N}$ ($({\mathcal A}_{j})_{1\leq j\leq N}$) is disjoint $(s,i)$-Li-Yorke chaotic, $(d,s,i)$-Li-Yorke chaotic in short.
\end{defn}

\begin{rem}\label{d-hyper-prc}
Assume that $(({\mathcal A}_{j,k})_{k\in
{\mathbb N}})_{1\leq j\leq N}$ is $d$-hypercyclic and $x$ is a corresponding $d$-hypercyclic vector for $(({\mathcal A}_{j,k})_{k\in
{\mathbb N}})_{1\leq j\leq N};$ see \cite[Definition 2.2]{faac} for the notion of $d{\mathcal F}$-hypercyclicity, here ${\mathcal F}$ is a collection of all non-empty subsets of ${\mathbb N}.$ Then it can be simply verified that $(({\mathcal A}_{j,k})_{k\in
{\mathbb N}})_{1\leq j\leq N}$ is $(d,\tilde{X},1,3)$-Li-Yorke chaotic with $\tilde{X}=span\{x\}.$
\end{rem}

Consider, in place of conditions \eqref{frechet-banach1}-\eqref{frechet-banach2}, the following ones with $\sigma>0$:
\begin{align}\label{bundainf3}
\underline{d}_{m_{n}}\Biggl( \bigcup_{ j\in {\mathbb N}_{N}} \bigl\{k \in {\mathbb N} :
d_{Y}\bigl(x_{j,k},y_{j,k}\bigr)< \sigma \bigr\}\Biggr)=0;
\end{align}
\begin{align}\label{bundainf4}
\bigl(\forall j\in {\mathbb N}_{N}\bigr)\, \ \underline{d}_{m_{n}}\Bigl(  \bigl\{k \in {\mathbb N} :
d_{Y}\bigl(x_{j,k},y_{j,k}\bigr)< \sigma \bigr\}\Bigr)=0.
\end{align}

Albeit not such important in our further investigations in comparision with Definition \ref{DC-unbounded-fric-DISJOINT-emenly}
and Definition \ref{DC-unbounded-fric-DISJOINT-emen-ly}, we will also introduce the following

\begin{defn}\label{fabrika}
Suppose that, for every $j\in {\mathbb N}_{N}$ and $k\in {\mathbb N},$ ${\mathcal A}_{j,k} : D({\mathcal A}_{j,k})\subseteq X \rightarrow Y$ is an MLO and $\tilde{X}$ is a linear
subspace of $X.$
Then we say that the sequence $(({\mathcal A}_{j,k})_{k\in
{\mathbb N}})_{1\leq j\leq N}$ is disjoint
$(\tilde{X},m_{n},3,1)$-Li-Yorke chaotic, $(d,\tilde{X},m_{n},3,1)$-Li-Yorke chaotic in short [disjoint
$(\tilde{X},m_{n},3,2)$-Li-Yorke chaotic, $(d,\tilde{X},m_{n},3,2)$-Li-Yorke chaotic in short], resp. disjoint
$(\tilde{X},m_{n},4,1)$-Li-Yorke chaotic, $(d,\tilde{X},m_{n},4,1)$-Li-Yorke chaotic in short [disjoint
$(\tilde{X},m_{n},4,2)$-Li-Yorke chaotic, $(d,\tilde{X},m_{n},4,2)$-Li-Yorke chaotic in short] iff there exist an uncountable
set $S\subseteq \bigcap_{j=1}^{N} \bigcap_{k=1}^{\infty} D({\mathcal A}_{j,k}) \cap \tilde{X}$ and $\sigma>0$ such that for each $\epsilon>0$ and for each pair $x,\
y\in S$ of distinct points we have that for each $j\in {\mathbb N}_{N}$ and $k\in {\mathbb N}$ there exist elements $x_{j,k}\in {\mathcal A}_{j,k}x$ and $y_{j,k}\in {\mathcal A}_{j,k}y$ such that
\eqref{bundainf3} and \eqref{bunda} [\eqref{bundainf3} and \eqref{bundaw}] hold, resp. \eqref{bundainf4} and \eqref{bunda} [\eqref{bundainf4} and \eqref{bundaw}] hold.

Let $s\in \{3,4\}$ and $i\in {\mathbb N}_{2}.$ Then we say that the sequence $(({\mathcal A}_{j,k})_{k\in
{\mathbb N}})_{1\leq j\leq N}$ is densely
$(d,\tilde{X},m_{n},s,i)$-Li-Yorke chaotic iff $S$ can be chosen to be dense in $\tilde{X}.$
A finite sequence $({\mathcal A}_{j})_{1\leq j\leq N}$ of MLOs on $X$ is said to be (densely)
$(d,\tilde{X},m_{n},s,i)$-distributionally chaotic\index{$(d,\tilde{X},s,i)$-Li-Yorke chaotic operator} iff the sequence $(({\mathcal A}_{j,k}\equiv
{\mathcal A}_{j}^{k})_{k\in {\mathbb N}})_{1\leq j\leq N}$ is.
The set $S$ is said to be $(d,\sigma_{\tilde{X}},m_{n},s,i)$-Li-Yorke scrambled set\index{ $(d,\sigma_{\tilde{X}},m_{n},s,i)$-Li-Yorke scrambled set} ($(d,\sigma,m_{n},s,i)$-Li-Yorke scrambled set in the case\index{(d,$\sigma,m_{n},s,i)$-Li-Yorke scrambled set}
that $\tilde{X}=X$) of $(({\mathcal A}_{j,k})_{k\in
{\mathbb N}})_{1\leq j\leq N}$ ($({\mathcal A}_{j})_{1\leq j\leq N}$);  in the case that
$\tilde{X}=X,$ then we also say that the sequence $(({\mathcal A}_{j,k})_{k\in
{\mathbb N}})_{1\leq j\leq N}$ ($({\mathcal A}_{j})_{1\leq j\leq N}$) is disjoint $(m_{n},s,i)$-Li-Yorke chaotic, $(d,m_{n},s,i)$-Li-Yorke chaotic in short.
\end{defn}

As in our previous investigations of disjoint $m_{n}$-distributional chaos, we need to know the minimal linear subspace $\tilde{X}$ for which the corresponding tuple of MLOs is $(d,\tilde{X},m_{n},s,i)$-Li-Yorke chaotic or $(d,\tilde{X},s,i)$-Li-Yorke chaotic, with the meaning clear. Since the $\tilde{X}$-Li-Yorke chaos and disjoint $\tilde{X}$-Li-Yorke chaos are rotation invariant, we essentially need to consider only such tuples of MLOs whose components are pairwise different and
which are not rotations of some other components in the tuple. It is also clear that the notion from Definition \ref{DC-unbounded-fric-DISJOINT-emenly} and Definition \ref{DC-unbounded-fric-DISJOINT-emen-ly} can be introduced for general binary relations between a topological space $X$ and a Fr\'echet space $Y$, as well as that the notion from Definition \ref{fabrika} can be introduced for general binary relations between a topological space $X$ and a pseudo-metric space $Y.$ For the sake of brevity, we will not consider disjoint (MLO) Li-Yorke extensions in this paper (see \cite{novascience} for similar problematic).

An idea of G. T. Pr\v ajitur\v a given on \cite[p. 690]{parij} can be used for construction of disjoint Li-Yorke chaotic MLOs. We will explain this idea only for $(d,\tilde{X},m_{n},1,i)$-Li-Yorke chaos:

\begin{example}\label{annyistyp}
Suppose that ${\mathcal A}_{j},\ {\mathcal B}_{j},\ {\mathcal C}_{j}$ are given MLOs in $X$, as well as the tuple $({\mathcal A}_{j})_{1\leq j\leq N}$ is $(d,\tilde{X},m_{n},1,i)$-Li-Yorke chaotic for some $(m_{n})\in {\mathrm R}.$ For each integer $j\in {\mathbb N}_{N},$
we define the multivalued map ${\mathcal T}_{j}\equiv \bigl(\begin{smallmatrix} {\mathcal A}_{j} &  {\mathcal B}_{j}\\ 0&  {\mathcal C}_{j}\end{smallmatrix}\bigr)$ by $D({\mathcal T}_{j}):=\{ (x,y) \in X\times X : x\in D({\mathcal A}_{j}),\ y\in D( {\mathcal B}_{j}) \cap D({\mathcal C}_{j})\}$ and ${\mathcal T}_{j}(x,y):=\{(z,\omega) \in X\times X : \omega \in {\mathcal C}_{j}y,\ \exists z_{1}\in  {\mathcal A}_{j}x,\ \exists z_{2}\in {\mathcal B}_{j}y, \ z=z_{1}+z_{2}\}.$ 
Then it can be easily seen that for each integer $j\in {\mathbb N}_{N},$ ${\mathcal T}_{j}$ is an MLO in $X\times X.$ Furthermore, it can be simply checked that the supposition $z_{j,k}\in {\mathcal A}_{j}^{k}z$ for some $k\in  {\mathbb N}$ and $j\in {\mathbb N}_{N}$ implies $(z_{j,k},0)\in {\mathcal T}_{j}^{k}(z,0).$ Observed this, it readily follows that
the tuple $({\mathcal T}_{j})_{1\leq j\leq N}$ is $(d,\tilde{X},m_{n},1,i)$-Li-Yorke chaotic in $X\times X,$ as well.
\end{example}

If the sequence $(({\mathcal A}_{j,k})_{k\in
{\mathbb N}})_{1\leq j\leq N}$ is (densely)
$\tilde{X}$-Li-Yorke chaotic in the sense of any notion introduced in the above three definitions, then for each $j\in {\mathbb N}_{N}$ we have that the sequence 
$({\mathcal A}_{j,k})_{k\in
{\mathbb N}}$ is (densely) $\tilde{X}$-Li-Yorke chaotic (in particular, if $\tilde{X}=Y=X$ and $(({\mathcal A}_{j,k})_{k\in
{\mathbb N}})_{1\leq j\leq N}$ is densely
Li-Yorke chaotic in the sense of any notion introduced above, then \cite[Proposition 4]{cheli} yields that for each $j\in {\mathbb N}_{N}$ one has
$\sigma_{p}({\mathcal A}_{j}^{\ast})\cap \{\lambda \in {\mathbb K} : |\lambda|\geq 1\}=\emptyset$). The converse statement does not hold even for orbits of linear continuous operators on Hilbert spaces, as the next example shows:

\begin{example}\label{count-grof-ly}
In \cite[Theorem 3.7]{countable}, Z. Yin, S. He and Y. Huang have shown that, for any two positive real numbers $a$ and $b$ such that $a<b,$ there exists an invertible operator $T$ acting on a Hilbert space $X$ such that $[a, b] =\{\lambda >0 : \lambda T \mbox{ is distributionally chaotic}\}$ and for any distinct values $\lambda_{1},\ \lambda_{2} \in [a, b],$ the operators $\lambda_{1}T$ and $\lambda_{2}T$ have no common Li-Yorke irregular vectors (see e.g. \cite[Definition 3]{countable} for the notion). Let $\lambda_{1}<\lambda_{2}$ and $\lambda_{1},\ \lambda_{2} \in [a, b].$ It can be easily checked that the operators $\lambda_{1}T$ and $\lambda_{2}T$ cannot
be $(d,X,2,4)$-Li-Yorke chaotic because any non-zero vector $z\in S-S,$ where $S$ denotes the corresponding scrambled set, needs to be a common Li-Yorke irregular vector for both operators $\lambda_{1}T$ and $\lambda_{2}T,$ as can be easily seen. This implies that $\lambda_{1}T$ and $\lambda_{2}T$ cannot be disjoint Li-Yorke chaotic in the sense of any notion introduced in Definition \ref{DC-unbounded-fric-DISJOINT-emenly} and Definition \ref{DC-unbounded-fric-DISJOINT-emen-ly}. Furthermore, these operators cannot be disjoint Li-Yorke chaotic in the sense of notion introduced in Definition \ref{fabrika} because, if we suppose the contrary, then for each non-zero vector $z\in S-S
$ there exist two strictly increasing sequences $(n_{k})$ and $(l_{k})$
of positive integers such that $\lim_{k\rightarrow \infty}\|(\lambda_{j}T)^{n_{k}}z\|=0$ and $\limsup_{k\rightarrow \infty}\|(\lambda_{j}T)^{l_{k}}z\|>0$ ($j=1,2$). By the proofs of \cite[Theorem 3.3, Theorem 3.7]{countable}, this would imply that there exists a constant $c(\lambda_{1},\lambda_{2}),$ independent of $z,$ such that $\|(\lambda_{1}T)^{n}z\|\leq c(\lambda_{1},\lambda_{2})\|z\|$ for all $n\in {\mathbb N}$ and therefore
$\|z\|\geq \sigma /c(\lambda_{1},\lambda_{2}).$ This is a contradiction because the set $S-S$ cannot be bounded away from zero.
\end{example}

Concerning Example \ref{count-grof-ly}, 
it should be noted that we have recently proved that the operators $\lambda_{1}T$ and $\lambda_{2}T$ are disjoint distributionally chaotic of type\\ $i\in \{4,5,6,8,9,10,11,12\};$ 
see \cite{ddc} for notion and more details. We will not analyze Li-Yorke analogues for these types of disjoint distributional chaos here.

We continue by providing one more illustrative example:

\begin{example}\label{count-grof-ly324} (see also \cite[Example 3.24]{ddc})
Consider
a weighted forward shift $F_{\omega} \in L(
l^{2}),$ defined by $F_{\omega} (x_{1}, x_{2}, \cdot \cdot \cdot) \mapsto (0, \omega_{1}x_{1},\omega_{2}x_{2},\cdot \cdot \cdot),$ where the sequence of weights $\omega =(\omega_{k})_{k\in {\mathbb N}}$ consists of sufficiently large blocks of $2$'s and blocks of $(1/2)$'s. Set $\sigma :=(1/\omega_{k})_{k\in {\mathbb N}}.$ Then for each non-zero vector $\langle x_{n}\rangle_{n\in {\mathbb N}}\in l^{2}$ there exists $n_{0}\in {\mathbb N}$ such that $x_{n_{0}}\neq 0$ and, for every integer $k\in {\mathbb N},$ we have 
\begin{align*}
\Bigl \|F_{\omega}^{k}\langle x_{n}\rangle_{n\in {\mathbb N}}+F_{\sigma}^{k}\langle x_{n}\rangle_{n\in {\mathbb N}}\Bigr\|&\geq 2|x_{n_{0}}|.
\end{align*}
This, in turn, implies that the operators $F_{\omega}$ and $F_{\sigma}$ cannot be  disjoint Li-Yorke chaotic in the sense of any notion introduced in Definition \ref{DC-unbounded-fric-DISJOINT-emenly}, as well as that $F_{\omega}$ and $F_{\sigma}$ cannot be $(d,\tilde{X},1,3)$-Li-Yorke chaotic [$(d,\tilde{X},2,3)$-Li-Yorke chaotic, $(d,\tilde{X},3,1)$-Li-Yorke chaotic, $(d,\tilde{X},4,1)$-Li-Yorke chaotic].
Now we will analyze the question when $F_{\omega}$ and $F_{\sigma}$ can be $(d,\tilde{X},1,4)$-Li-Yorke chaotic or $(d,\tilde{X},2,4)$-Li-Yorke chaotic. Assume that $e_{1}$ is a Li-Yorke irregular vector for $F_{\omega}.$ Then $e_{1}$ is a Li-Yorke irregular vector for $F_{\sigma}$ and it trivially follows that the operators $F_{\omega}$ and $F_{\sigma}$ are $(d,span\{e_{1}\},2,4)$-Li-Yorke chaotic with $S=span\{e_{1}\}$ being the corresponding disjoint scrambled set. Similarly, if $e_{1}$ is an $m_{n}$-distributionally irregular vector for $F_{\omega},$ then $e_{1}$ is likewise an $m_{n}$-distributionally irregular vector for $F_{\sigma},$ and the operators $F_{\omega}$ and $F_{\sigma}$ are $(d,span\{e_{1}\},m_{n},4,2)$-Li-Yorke chaotic with $S=span\{e_{1}\}$ being the corresponding disjoint scrambled set (see \cite{ddcNS} for the notion).
\end{example}

Observe that for each subset $A\subseteq {\mathbb N}$ we have that the assumption
$\underline{d}_{m_{n}}(A)=0$ for some $(m_{n})\in {\mathrm R}$ implies $\underline{Bd}_{l;m_{n}}(A)=0$ (if we suppose the contrary, the set $A$ needs to be syndetic \cite{F-operatori} and therefore $\underline{d}_{m_{n}}(A)>0,$ which contradicts our assumption). Therefore, the validity of \eqref{bunda1} implies that of \eqref{bunda2} and we can trivially verify that 
the following proposition holds good:

\begin{prop}\label{beqa}
Suppose that, for every $j\in {\mathbb N}_{N}$ and $k\in {\mathbb N},$ ${\mathcal A}_{j,k} : D({\mathcal A}_{j,k})\subseteq X \rightarrow Y$ is an {\em MLO,} $(m_{n})\in {\mathrm R}$ and $\tilde{X}$ is a linear
subspace of $X.$ Then we have:
\begin{itemize}
\item[(i)] The sequence $(({\mathcal A}_{j,k})_{k\in
{\mathbb N}})_{1\leq j\leq N}$ is $(d,\tilde{X},m_{n},2,i)$-Li-Yorke chaotic if it is $(d,\tilde{X},m_{n},1,i)$-Li-Yorke chaotic ($i=1,2$).
\item[(ii)] The sequence $(({\mathcal A}_{j,k})_{k\in
{\mathbb N}})_{1\leq j\leq N}$ is $(d,\tilde{X},m_{n},s,2)$-Li-Yorke chaotic if it is $(d,\tilde{X},m_{n},s,1)$-Li-Yorke chaotic ($s=1,2$); if 
the sequence $(({\mathcal A}_{j,k})_{k\in
{\mathbb N}})_{1\leq j\leq N}$ is $(d,\tilde{X},m_{n},s,2)$-Li-Yorke chaotic, then it is $(d,\tilde{X},s,3)$-Li-Yorke chaotic ($s=1,2$).
\item[(iii)] The sequence $(({\mathcal A}_{j,k})_{k\in
{\mathbb N}})_{1\leq j\leq N}$ is $(d,\tilde{X},s,4)$-Li-Yorke chaotic if it is\\ $(d,\tilde{X},m_{n},s,3)$-Li-Yorke chaotic ($s=3,4$).
\item[(iv)] The sequence $(({\mathcal A}_{j,k})_{k\in
{\mathbb N}})_{1\leq j\leq N}$ is $(d,\tilde{X},m_{n},4,i)$-Li-Yorke chaotic if it is $(d,\tilde{X},m_{n},3,i)$-Li-Yorke chaotic ($i=1,2$).
\item[(v)] The sequence $(({\mathcal A}_{j,k})_{k\in
{\mathbb N}})_{1\leq j\leq N}$ is $(d,\tilde{X},m_{n},s,2)$-Li-Yorke chaotic if it is $(d,\tilde{X},m_{n},s,1)$-Li-Yorke chaotic ($s=3,4$).
\item[(vi)] The sequence $(({\mathcal A}_{j,k})_{k\in
{\mathbb N}})_{1\leq j\leq N}$ is $(d,\tilde{X},n,s,i)$-Li-Yorke chaotic if it is $(d,\tilde{X},m_{n},s,i)$-Li-Yorke chaotic ($s\in {\mathbb N}_{4},$ $i=1,2$).
\end{itemize}
\end{prop}

For orbits of linear continuous operators in Banach spaces, it is worth noting that the following equivalence relations hold:

\begin{prop}\label{banach-space}
Suppose that $X$ is a Banach space and $T_{j}\in L(X)$ for all $j\in {\mathbb N}_{N}.$ Then the following statements are mutually equivalent: 
\begin{itemize}
\item[(i)] There exist two strictly increasing sequences $(l_{k})$ and $(s_{k})$ of positive integers and vector $x\in X$ such that $\lim_{k\rightarrow \infty}T_{j}^{s_{k}}x=0$ and $\lim_{k\rightarrow \infty}\|T_{j}^{l_{k}}x\|=+\infty$ ($j\in {\mathbb N}_{N}$).
\item[(ii)] For every $\sigma>0,$ $\epsilon>0$ and $(m_{n}) \in {\mathrm R},$ we have that
\begin{align}\label{jednacina-kurtaz}
\begin{split}
& \underline{Bd}_{l;m_{n}}\Biggl( \bigcup_{ j\in {\mathbb N}_{N}} \bigl\{k \in {\mathbb N} :
\| T_{j}^{k}x\|< \sigma \bigr\}\Biggr)=0,\mbox{ and }
\\
& \underline{Bd}_{l;m_{n}}\Biggl( \bigcup_{j\in {\mathbb N}_{N}} \bigl\{k \in {\mathbb N} : \| T_{j}^{k}x\|
\geq \epsilon \bigr\}\Biggr)=0.
\end{split}
\end{align} 
\item[(iii)] For every $\sigma>0,$ $\epsilon>0$ we have that \eqref{jednacina-kurtaz} holds with $m_{n}\equiv n.$
\end{itemize} 
\end{prop}

\begin{proof}
The only non-trivial is to show that (i) implies (ii); see also the proof of \cite[Proposition 2.16(i)]{reit}. So, let $\sigma>0,$ $\epsilon>0$ and $(m_{n}) \in {\mathrm R}$ be fixed. By definition of $\underline{Bd}_{l;m_{n}}(\cdot),$
it suffices to prove that for each fixed number $s>0$ one has:
\begin{align}\label{dziber}
\liminf_{n\rightarrow \infty}\frac{\Bigl| \bigcup_{j\in {\mathbb N}_{N}}\bigl\{k\in {\mathbb N} : \|T_{j}^{k}x\|<\sigma \bigr\} \cap [n+1,n+m_{s}] \Bigr| }{s}=0
\end{align}
and
\begin{align}\label{dziber1}
\liminf_{n\rightarrow \infty}\frac{\Bigl| \bigcup_{j\in {\mathbb N}_{N}}\bigl\{k\in {\mathbb N} : \|T_{j}^{k}x\|\geq \epsilon \bigr\} \cap [n+1,n+m_{s}] \Bigr| }{s}=0.
\end{align}
It is clear that there exist two strictly increasing sequences of positive integers $(l_{k}')$ and $(j_{k}')$ with unbounded differences such that 
$\| T_{j}^{l_{k}'}x\|<\sigma(2+\|T_{1}\|+\cdot \cdot \cdot +\|T_{N}\|)^{-k^{2}-m_{s}-1}/2$ and $\| T^{j_{k}'}_{j}x\|>2\sigma(2+\|T_{1}\|+\cdot \cdot \cdot +\|T_{N}\|)^{k^{2}+m_{s}+1}$ for all $k\in {\mathbb N}$ and $j\in {\mathbb N}_{N}.$
An elementary line of reasoning shows that the sets $\bigcup_{j\in {\mathbb N}_{N}}\{k\in {\mathbb N} : \|T^{k}_{j}x\|>\sigma \} \cap [l_{k}',l_{k}'+\lceil m_{s} \rceil] $ and 
$\bigcup_{j\in {\mathbb N}_{N}}\{k\in {\mathbb N} : \|T_{j}^{k}x\|<\sigma \} \cap [j_{k}'-\lceil m_{s}\rceil , j_{k}'] $ are empty, finishing the proofs of \eqref{dziber}-\eqref{dziber1}.
\end{proof}

\subsection{Disjoint Li-Yorke irregular vectors and manifolds}\label{emojstevo}

For any type of disjoint Li-Yorke chaos introduced above, we can define corresponding notion of disjoint semi-Li-Yorke irregular vectors and disjoint Li-Yorke irregular vectors.
Consider the following conditions:
\begin{align}\label{frechet-banach1-vector}
\mbox{the same as }\eqref{frechet-banach1}\mbox{ with the term }x_{j,l_{k}}-y_{j,l_{k}}
\mbox{ replaced therein  with }x_{j,l_{k}};
\end{align}
\begin{align}
\notag & \mbox{the same as }\eqref{frechet-banach1}\mbox{ with the terms }x_{j,l_{k}}-y_{j,l_{k}}\mbox{ and }\lim_{k\rightarrow \infty}p_{m}^{Y}\bigl(x_{j,l_{k}}-y_{j,l_{k}}\bigr)=+\infty
\\\label{frechet-banach1-semivector} & \mbox{ replaced therein  with }x_{j,l_{k}}\mbox{ and }\lim_{k\rightarrow \infty}p_{m}^{Y}\bigl(x_{j,l_{k}}\bigr)>0,\mbox{ respectively};
\end{align}
\begin{align}\label{frechet-banach2-vector}
\mbox{the same as }\eqref{frechet-banach2}\mbox{ with the term }x_{j,l_{k}^{j}}-y_{j,l_{k}^{j}}
\mbox{ replaced therein  with }x_{j,l_{k}^{j}};
\end{align}
\begin{align}
\notag & \mbox{the same as }\eqref{frechet-banach2}\mbox{ with the terms }x_{j,l_{k}^{j}}-y_{j,l_{k}^{j}}\mbox{ and }\lim_{k\rightarrow \infty}p_{m}^{Y}\bigl(x_{j,l_{k}^{j}}-y_{j,l_{k}^{j}}\bigr)=+\infty
\\\label{frechet-banach2-semivector} & \mbox{ replaced therein  with }x_{j,l_{k}^{j}}\mbox{ and }\lim_{k\rightarrow \infty}p_{m}^{Y}\bigl(x_{j,l_{k}^{j}}\bigr)>0,\mbox{ respectively};
\end{align}
\begin{align}
\notag \mbox{the same as }&\eqref{bunda}\mbox{ with the term }\lim_{k\rightarrow \infty}d_{Y}\bigl(x_{j,n_{k}},y_{j,n_{k}}\bigr)=0
\\\label{bunda-vector} & \mbox{ replaced therein  with }\lim_{k\rightarrow \infty}d_{Y}\bigl(x_{j,n_{k}},0\bigr)=0;
\end{align}
\begin{align}
\notag \mbox{the same as }&\eqref{bundaw}\mbox{ with the term }\lim_{k\rightarrow \infty}d_{Y}\bigl(x_{j,n_{k}},y_{j,n_{k}}\bigr)=0
\\\label{bundaw-vector} & \mbox{ replaced therein  with }\lim_{k\rightarrow \infty}d_{Y}\bigl(x_{j,n_{k}},0\bigr)=0.
\end{align}

Now we are ready to introduce the following notion:

\begin{defn}\label{DC-unbounded-fric-DISJOINT-emenlyprc}
Let $i\in {\mathbb N}_{2}$ and $(m_{n})\in {\mathrm R}.$
Suppose that, for every $j\in {\mathbb N}_{N}$ and $k\in {\mathbb N},$ ${\mathcal A}_{j,k} : D({\mathcal A}_{j,k})\subseteq X \rightarrow Y$ is an MLO, $\tilde{X}$ is a linear
subspace of $X,$ and $x\in \bigcap_{j\in {\mathbb N}_{N}}\bigcap_{k\in {\mathbb N}}D({\mathcal A}_{j,k}) \cap \tilde{X}.$
Then we say that:
\begin{itemize}
\item[(i)] $x$ is $(d,\tilde{X},m_{n},1,1)$-Li-Yorke irregular vector for $(({\mathcal A}_{j,k})_{k\in
{\mathbb N}})_{1\leq j\leq N}$ iff $x$ is $(d,m_{n})$-distributionally near to zero of type $1$ for $(({\mathcal A}_{j,k})_{k\in
{\mathbb N}})_{1\leq j\leq N}$ and \eqref{frechet-banach1-vector} holds with elements $x_{j,l_{k}}\in {\mathcal A}_{j,l_{k}}x;$
\item[(ii)] $x$ is $(d,\tilde{X},m_{n},1,1)$-Li-Yorke semi-irregular vector for $(({\mathcal A}_{j,k})_{k\in
{\mathbb N}})_{1\leq j\leq N}$ iff $x$ is $(d,m_{n})$-distributionally near to zero of type $1$ for $(({\mathcal A}_{j,k})_{k\in
{\mathbb N}})_{1\leq j\leq N}$ and \eqref{frechet-banach1-semivector} holds  with elements $x_{j,l_{k}}\in {\mathcal A}_{j,l_{k}}x;$
\item[(iii)] $x$ is $(d,\tilde{X},m_{n},1,2)$-Li-Yorke irregular vector for $(({\mathcal A}_{j,k})_{k\in
{\mathbb N}})_{1\leq j\leq N}$ iff $x$ is $(d,\tilde{X},m_{n},1,1)$-Li-Yorke irregular vector for $(({\mathcal A}_{j,k})_{k\in
{\mathbb N}})_{1\leq j\leq N}$ iff $x$ is reiteratively $(d,m_{n})$-distributionally near to zero of type $1$ for $(({\mathcal A}_{j,k})_{k\in
{\mathbb N}})_{1\leq j\leq N}$ and \eqref{frechet-banach1-vector} holds  with elements $x_{j,l_{k}}\in {\mathcal A}_{j,l_{k}}x;$
\item[(iv)] $x$ is $(d,\tilde{X},m_{n},1,2)$-Li-Yorke irregular vector for $(({\mathcal A}_{j,k})_{k\in
{\mathbb N}})_{1\leq j\leq N}$ iff $x$ is reiteratively $(d,m_{n})$-distributionally near to zero of type $1$ for $(({\mathcal A}_{j,k})_{k\in
{\mathbb N}})_{1\leq j\leq N}$ and \eqref{frechet-banach1-semivector} holds  with elements $x_{j,l_{k}}\in {\mathcal A}_{j,l_{k}}x;$
\item[(v)] $x$ is $(d,\tilde{X},m_{n},2,1)$-Li-Yorke irregular vector for $(({\mathcal A}_{j,k})_{k\in
{\mathbb N}})_{1\leq j\leq N}$ iff $x$ is $(d,m_{n})$-distributionally near to zero of type $1$ for $(({\mathcal A}_{j,k})_{k\in
{\mathbb N}})_{1\leq j\leq N}$ and \eqref{frechet-banach2-vector} holds with elements $x_{j,l_{k}^{j}}\in {\mathcal A}_{j,l_{k}^{j}}x;$
\item[(vi)] $x$ is $(d,\tilde{X},m_{n},2,1)$-Li-Yorke semi-irregular vector for $(({\mathcal A}_{j,k})_{k\in
{\mathbb N}})_{1\leq j\leq N}$ iff $x$ is $(d,m_{n})$-distributionally near to zero of type $1$ for $(({\mathcal A}_{j,k})_{k\in
{\mathbb N}})_{1\leq j\leq N}$ and \eqref{frechet-banach2-semivector} holds  with elements $x_{j,l_{k}^{j}}\in {\mathcal A}_{j,l_{k}^{j}}x;$
\item[(vii)] $x$ is $(d,\tilde{X},m_{n},2,2)$-Li-Yorke irregular vector for $(({\mathcal A}_{j,k})_{k\in
{\mathbb N}})_{1\leq j\leq N}$ iff $x$ is $(d,\tilde{X},m_{n},1,1)$-Li-Yorke irregular vector for $(({\mathcal A}_{j,k})_{k\in
{\mathbb N}})_{1\leq j\leq N}$ iff $x$ is reiteratively $(d,m_{n})$-distributionally near to zero of type $1$ for $(({\mathcal A}_{j,k})_{k\in
{\mathbb N}})_{1\leq j\leq N}$ and \eqref{frechet-banach2-vector} holds  with elements $x_{j,l_{k}^{j}}\in {\mathcal A}_{j,l_{k}^{j}}x;$
\item[(viii)] $x$ is $(d,\tilde{X},m_{n},2,2)$-Li-Yorke semi-irregular vector for $(({\mathcal A}_{j,k})_{k\in
{\mathbb N}})_{1\leq j\leq N}$ iff $x$ is reiteratively $(d,m_{n})$-distributionally near to zero of type $1$ for $(({\mathcal A}_{j,k})_{k\in
{\mathbb N}})_{1\leq j\leq N}$ and \eqref{frechet-banach2-semivector} holds  with elements $x_{j,l_{k}^{j}}\in {\mathcal A}_{j,l_{k}^{j}}x$.
\end{itemize}
\end{defn}
 
\begin{defn}\label{DC-unbounded-fric-DISJOINT-emenly-fric}
Let $i\in \{3,4\}.$ 
Suppose that, for every $j\in {\mathbb N}_{N}$ and $k\in {\mathbb N},$ ${\mathcal A}_{j,k} : D({\mathcal A}_{j,k})\subseteq X \rightarrow Y$ is an MLO, $\tilde{X}$ is a linear
subspace of $X,$ and $x\in \bigcap_{j\in {\mathbb N}_{N}}\bigcap_{k\in {\mathbb N}}D({\mathcal A}_{j,k}) \cap \tilde{X}.$
Then we say that:
\begin{itemize}
\item[(i)] $x$ is $(d,\tilde{X},1,3)$-Li-Yorke irregular vector for $(({\mathcal A}_{j,k})_{k\in
{\mathbb N}})_{1\leq j\leq N}$ iff \eqref{bunda-vector} and \eqref{frechet-banach1-vector} hold with elements $x_{j,l_{k}}\in {\mathcal A}_{j,l_{k}}x;$
\item[(ii)] $x$ is $(d,\tilde{X},1,3)$-Li-Yorke semi-irregular vector for $(({\mathcal A}_{j,k})_{k\in
{\mathbb N}})_{1\leq j\leq N}$ iff \eqref{bunda-vector} and \eqref{frechet-banach1-semivector} hold with elements $x_{j,l_{k}}\in {\mathcal A}_{j,l_{k}}x;$
\item[(iii)] $x$ is $(d,\tilde{X},1,4)$-Li-Yorke irregular vector for $(({\mathcal A}_{j,k})_{k\in
{\mathbb N}})_{1\leq j\leq N}$ iff \eqref{bundaw-vector} and \eqref{frechet-banach1-vector} hold with elements $x_{j,l_{k}}\in {\mathcal A}_{j,l_{k}}x;$
\item[(iv)] $x$ is $(d,\tilde{X},1,4)$-Li-Yorke irregular vector for $(({\mathcal A}_{j,k})_{k\in
{\mathbb N}})_{1\leq j\leq N}$ iff  \eqref{bundaw-vector} and \eqref{frechet-banach1-semivector} hold  with elements $x_{j,l_{k}}\in {\mathcal A}_{j,l_{k}}x;$
\item[(v)] $x$ is $(d,\tilde{X},2,3)$-Li-Yorke irregular vector for $(({\mathcal A}_{j,k})_{k\in
{\mathbb N}})_{1\leq j\leq N}$ iff \eqref{bunda-vector} and \eqref{frechet-banach2-vector} hold with elements $x_{j,l_{k}^{j}}\in {\mathcal A}_{j,l_{k}^{j}}x$;
\item[(vi)] $x$ is $(d,\tilde{X},2,3)$-Li-Yorke semi-irregular vector for $(({\mathcal A}_{j,k})_{k\in
{\mathbb N}})_{1\leq j\leq N}$ iff \eqref{bunda-vector} and \eqref{frechet-banach2-semivector} hold with elements $x_{j,l_{k}^{j}}\in {\mathcal A}_{j,l_{k}^{j}}x$;
\item[(vii)] $x$ is $(d,\tilde{X},2,4)$-Li-Yorke irregular vector for $(({\mathcal A}_{j,k})_{k\in
{\mathbb N}})_{1\leq j\leq N}$ iff \eqref{bundaw-vector} and \eqref{frechet-banach2-vector} hold with elements $x_{j,l_{k}^{j}}\in {\mathcal A}_{j,l_{k}^{j}}x$;
\item[(viii)] $x$ is $(d,\tilde{X},2,4)$-Li-Yorke semi-irregular vector for $(({\mathcal A}_{j,k})_{k\in
{\mathbb N}})_{1\leq j\leq N}$ iff \eqref{bundaw-vector} and \eqref{frechet-banach2-semivector} hold with elements $x_{j,l_{k}^{j}}\in {\mathcal A}_{j,l_{k}^{j}}x$.
\end{itemize}
\end{defn}

Let $\{0\} \neq X' \subseteq \tilde{X}$ be a linear manifold. 
\begin{itemize}
\item[d1.] Suppose $i,\ j\in {\mathbb N}_{2}$  and $(m_{n})\in {\mathrm R}.$ Then
we say that
 $X'$ is a $(d,\tilde{X},m_{n},i,j)$-Li-Yorke (semi-)irregular manifold\index{$(d,\tilde{X},m_{n},i,j)$-Li-Yorke
(semi-)irregular manifold} for $(({\mathcal A}_{j,k})_{k\in {\mathbb N}})_{1\leq j\leq N}$
($(d,m_{n},i,j)$-Li-Yorke (semi-)irregular manifold\index{Li-Yorke (semi-)irregular manifold} in the case that $\tilde{X}=X$)
iff any element $x\in (X' \cap
\bigcap_{j=1}^{N}\bigcap_{k=1}^{\infty}D({\mathcal A}_{j,k})) \setminus \{0\}$ is a
$(d,\tilde{X},m_{n},i,j)$-Li-Yorke (semi-)irregular vector for
$(({\mathcal A}_{j,k})_{k\in {\mathbb N}})_{1\leq j\leq N};$ the notion of a ($(d,m_{n},i,j)$-, $(d,\tilde{X},m_{n},i,j)$-)Li-Yorke (semi-)irregular manifold for $({\mathcal A}_{j})_{1\leq j\leq N}$ is defined similarly.
\item[d2.] Suppose $i\in {\mathbb N}_{2}$ and $j\in \{3,4\}.$ Then
we say that
 $X'$ is a $(d,\tilde{X},i,j)$-Li-Yorke (semi-)irregular manifold\index{$(d,\tilde{X},i,j)$-Li-Yorke
(semi-)irregular manifold} for $(({\mathcal A}_{j,k})_{k\in {\mathbb N}})_{1\leq j\leq N}$
($(d,i,j)$-Li-Yorke (semi-)irregular manifold\index{Li-Yorke (semi-)irregular manifold} in the case that $\tilde{X}=X$)
iff any element $x\in (X' \cap
\bigcap_{j=1}^{N}\bigcap_{k=1}^{\infty}D({\mathcal A}_{j,k})) \setminus \{0\}$ is a
$(d,\tilde{X},i,j)$-Li-Yorke (semi-)irregular vector for
$(({\mathcal A}_{j,k})_{k\in {\mathbb N}})_{1\leq j\leq N};$ the notion of a ($(d,i,j)$-, $(d,\tilde{X},i,j)$-)Li-Yorke (semi-)irregular manifold for $({\mathcal A}_{j})_{1\leq j\leq N}$ is defined similarly.
\end{itemize}
We have the following:
\begin{itemize}
\item[d3.]
Suppose that $i,\ j\in {\mathbb N}_{2},$ $(m_{n})\in {\mathrm R}$ and $0\neq x\in \tilde{X} \cap \bigcap_{j=1}^{N}\bigcap_{k=1}^{\infty}D({\mathcal A}_{j,k})$ is a
$(d,\tilde{X},m_{n},i,j)$-Li-Yorke (semi-)irregular vector for $(({\mathcal A}_{j,k})_{k\in {\mathbb N}})_{1\leq j\leq N}.$
Then $X'\equiv span\{x\}$
is a $(d,\tilde{X},m_{n},i,j)$-Li-Yorke (semi-)irregular manifold for
$(({\mathcal A}_{j,k})_{k\in {\mathbb N}})_{1\leq j\leq N};$
\item[d4.]
Suppose $i\in {\mathbb N}_{2},$ $j\in \{3,4\}$ and $0\neq x\in \tilde{X} \cap \bigcap_{j=1}^{N}\bigcap_{k=1}^{\infty}D({\mathcal A}_{j,k})$ is a
$(d,\tilde{X},i,j)$-Li-Yorke (semi-)irregular vector for $(({\mathcal A}_{j,k})_{k\in {\mathbb N}})_{1\leq j\leq N}.$
Then $X'\equiv span\{x\}$
is a $(d,\tilde{X},i,j)$-Li-Yorke (semi-)irregular manifold for\\
$(({\mathcal A}_{j,k})_{k\in {\mathbb N}})_{1\leq j\leq N}.$
\end{itemize}
If $X'$ is dense in $\tilde{X},$
then the notions of dense ($(d,m_{n},i,j)$-, $(d,\tilde{X},m_{n},i,j)$-)Li-Yorke
(semi-)irregular manifolds, $(d,i,j)$-, $(d,\tilde{X},i,j)$-)Li-Yorke
(semi-)irregular manifolds, etc., are defined analogically.

It can be simply verified by a great number of concrete and very plain examples that the notions of 
$(d,\tilde{X},m_{n},i,j)$-Li-Yorke chaos and $(d,\tilde{X},m_{n},i_{1},j_{1})$-Li-Yorke chaos differ
if  $(m_{n})\in {\mathrm R},$ $i,\ i_{1},\ j,\ j_{1}\in {\mathbb N}_{2}$ and $(i,j)\neq (i_{1},j_{1}),$ as well as that the notions of 
$(d,\tilde{X},i,j)$-Li-Yorke chaos and $(d,\tilde{X},i_{1},j_{1})$-Li-Yorke chaos differ
if $i,\ i_{1}\in {\mathbb N}_{2},$ $j,\ j_{1}\in \{3,4\}$ and $(i,j)\neq (i_{1},j_{1}).$
The counterexamples exist even for general sequences of linear continuous operators on finite-dimensional spaces, for which it is also clear that they can have $(d,\tilde{X},m_{n},i,j)$-Li-Yorke semi-irregular vectors but not any
$(d,\tilde{X},m_{n},i,j)$-Li-Yorke irregular vector (take, for example, $X:=Y:={\mathbb K}^{n},$ $T_{j}:=0$ for even $j'$s and $T_{j}:=2I$ for odd $j'$s). Two real problems are: 

\begin{prob}\label{pizdatistrina}
\begin{itemize}
\item[(i)] Let $(m_{n})\in {\mathrm R}$ and $i,\ j\in {\mathbb N}_{2}.$
Does there exist a tuple $(T_{j})_{1\leq j\leq N}$ of linear continuous operators 
on an infinite-dimensional space $X$ admitting a $(d,\tilde{X},m_{n},i,j)$-Li-Yorke semi-irregular vector and its neighborhood which does not contain any $(d,\tilde{X},m_{n},i,j)$-Li-Yorke irregular vector for $(T_{j})_{1\leq j\leq N}$?
\item[(ii)] Let $i\in {\mathbb N}_{2}$ and $j\in \{3,4\}.$ Does there exist a tuple $(T_{j})_{1\leq j\leq N}$ of linear continuous operators 
on an infinite-dimensional space $X$ admitting a $(d,\tilde{X},i,j)$-Li-Yorke semi-irregular vector and its neighborhood which does not contain any $(d,\tilde{X},i,j)$-Li-Yorke irregular vector for $(T_{j})_{1\leq j\leq N}$?
\end{itemize}
\end{prob}

In connection with the above problems, we would like to note that for a continuous linear operator $T\in L(X)$ any neighborhood of a $\tilde{X}$-Li-Yorke semi-irregular vector for $T$
contains a $\tilde{X}$-irregular vector for $T ,$ provided that $\{T^{j}x : j\in {\mathbb N}_{0}\}\subseteq \tilde{X}$ (see \cite[Lemma 7, Theorem 8]{band} and \cite[Definition 2.2]{nis-dragan}
for the notion, as well as
\cite[Lemma 3.5, Remark 3.6, Theorem 3.7]{nis-dragan} for a continuous analogue). Unfortunately, the proof of   
\cite[Lemma 7]{band} cannot be recovered for disjointness, to our best knowledge, and we must follow some other approaches for solving Problem \ref{pizdatistrina}. It is also clear that we can raise a continuous counterpart of Problem \ref{pizdatistrina} for strongly continuous semigroups of operators (cf. also the proof of implication \cite[Theorem 2.2, (1-2) $\Rightarrow$ (1-1)]{xwu}).

Concerning the notion introduced in Definition \ref{fabrika}, the corresponding notion of disjoint Li-Yorke irregular vectors and (uniform) disjoint Li-Yorke irregular manifolds can be 
also accompanied. The main difference is the use of notion of $(d,m_{n})$-distributionally $m$-unbounded vectors of type $1$ for $(({\mathcal A}_{j,k})_{k\in {\mathbb N}})_{1\leq j\leq N}$
and $m_{n}$-distributionally $m$-unbounded vectors for $(({\mathcal A}_{k})_{k\in {\mathbb N}})$ in place of conditions analyzed in the equations \eqref{frechet-banach1-vector} or \eqref{frechet-banach2-vector}; see \cite{ddcNS}
for more details. For the sake of brevity and better exposition, we will skip all related details about this subject.

In our previous research studies, we have observed some important differences between Banach spaces and Fr\' echet spaces concerning the existence of (disjoint) $m_{n}$-distributionally unbounded vectors. These differences are also perceived for (disjoint) Li-Yorke chaos (cf. \cite{ddc} for more details):

\begin{example}\label{gos}
Set $\tilde{B}:=\{k\in {\mathbb N} : {\mathcal A}_{j,k}\mbox{ is purely multivalued for all }j\in {\mathbb N}_{N}\}.$ Let $Y$ be a Banach space and let $\tilde{B}$ be infinite.  
Then any non-zero vector $x\in \bigcap_{j=1}^{N}\bigcap_{k=1}^{\infty}D({\mathcal A}_{j,k})$ satisfies \eqref{frechet-banach1-vector}, which is no longer true 
in the case that $Y$ is a Fr\' echet space.
\end{example}

\section{The proof and corollaries of main result}\label{main-reza}

We start this section by inserting the proof of Theorem \ref{na-dobro1-ly}:\vspace{0.1cm}

\noindent {\it Proof of Theorem \ref{na-dobro1-ly}.} The proof is very similar to that of \cite[Theorem 15]{2013JFA} and we will only outline the main details. It suffices to consider the case in which 
$X$ and $Y$ are Fr\' echet spaces whose topology is induced by a countable system of seminorms because otherwise we can endow $Y$  (or $X$, if it is a Banach space) with the following
increasing family of seminorms $p_{n}^{Y}(y):=n\|y\|_{Y}$ ($n\in {\mathbb N},$
$y\in Y$), which turns the space $Y$ into a linearly and topologically homeomorphic
Fr\' echet space. So, let it be the case. Then
it is clear that, for
every $j\in {\mathbb N}_{N}$ and $l,\ k\in {\mathbb N},
$ there exist finite numbers $c_{j,l,k}>0$ and $a_{j,l,k}\in {\mathbb N}$ such that $p_{l}^{Y}(T_{j,k}x)\leq c_{j,k,l}p_{a_{j,k,l}}(x),$ $x\in X,$ $k,\ l\in {\mathbb N},$ $j\in {\mathbb N}_{N}.$
Introducing recursively the following fundamental system of increasing seminorms $p_{n}'(\cdot)$ ($n\in {\mathbb N}$) on $X:$
\begin{align*}
& p_{1}'(x)\equiv  p_{1}(x),\quad x\in X,\\
& p_{2}'(x)\equiv \sum_{j=1}^{N}\bigl[p_{1}'(x)+c_{j,1,1}p_{a_{j,1,1}}(x)+p_{2}(x)\bigr],\quad x\in X,\\
& \cdot \cdot \cdot \\
& p_{n+1}'(x)\equiv \sum_{j=1}^{N}\bigl[p_{n}'(x)+c_{j,1,n}p_{a_{j,1,n}}(x)+\cdot \cdot \cdot +c_{j,n,1}p_{a_{j,n,1}}(x)+p_{n+1}(x)\bigr],\quad x\in X,\\
& \cdot \cdot \cdot ,
\end{align*}
we may assume without loss of generality that 
\begin{align}\label{grozno}
p_{Y}^{l}(T_{j,k}x)\leq p_{k+l}(x)\ \ \mbox{for all }\ \ x\in X,\ \ j\in {\mathbb N}_{N}\mbox{ and }\ \ k,\ l\in {\mathbb N}.
\end{align} 
Furthermore, we may assume without loss of generality that 
$m=1.$
Then we can construct a sequence $(x_{l})_{l\in {\mathbb N}}$ in $X_{0}$ and a strictly increasing sequence $(k_{l})_{k\in {\mathbb N}}$ of positive integers such that, for every $l\in {\mathbb N},$ one has: $p_{l}(x_{l})\leq 1,$ 
$p_{1}^{Y}(T_{j,n_{k_{l}}}x_{l})>l2^{l}$
and $ p_{Y}^{l}(T_{j,k}x_{s})<1/l ,$ provided $j\in {\mathbb N}_{N},$ $ s=1,\cdot \cdot \cdot,l-1
$ and $k\geq m_{k_{l}+1}/l.$
Take any strictly increasing sequence $(r_{q})_{q\in {\mathbb N}}$  in ${\mathbb N} \setminus \{1\}$
such that 
\begin{align}\label{miruga}
r_{q+1}\geq 1+r_{q}+m_{n_{k_{r_{q}+1}}}+n_{k_{r_{q}+1}}\mbox{ for all }q\in {\mathbb N}.
\end{align} 
Let $\alpha \in \{0,1\}^{\mathbb N}$ be a sequence defined by $\alpha_{s}=1$ iff $s=r_{q}$ for some $q\in {\mathbb N}.$
Further on, let $\beta \in \{0,1\}^{\mathbb N}$ 
contains an infinite number of $1'$s and let $\beta_{q}\leq \alpha_{q}$ for all $q\in {\mathbb N}.$ If $\beta_{r_{q_{0}}}=1$ for some $q_{0}\in {\mathbb N}$ and $x_{\beta}=\sum_{q=1}^{\infty}\beta_{r_{q}}x_{r_{q}}/2^{r_{q}},$ then with $k=n_{k_{r_{q_{0}}}}$ 
and $j\in {\mathbb N}_{N},$ we have  $1+k= 1+n_{k_{r_{q_{0}}}} \leq r_{q_{0}+1}$ for $q>q_{0},$ $k\geq m_{k_{r_{q}}}$ for $q<q_{0}$
and $p_{Y}^{r_{q}+1}(T_{j,k}x_{s})<1/(r_{q}+1)$ for $s<r_{q_{0}},$ 
as well as:
\begin{align*}
p_{Y}^{1}\bigl( T_{j,k}x_{\beta}\bigr) & \geq r_{q_{0}}-\sum_{q<q_{0}}\frac{p_{Y}^{1}(T_{j,k}x_{r_{q}})}{2^{r_{q}}} -\sum_{q>q_{0}}\frac{p_{Y}^{1}(T_{j,k}x_{r_{q}})}{2^{r_{q}}}
\\ & \geq r_{q_{0}}-\sum_{q<q_{0}}\frac{p_{Y}^{1}(T_{j,k}x_{r_{q}})}{2^{r_{q}}}-\sum_{q>q_{0}}\frac{p_{1+k}(x_{r_{q}})}{2^{r_{q}}}
\\ & \geq r_{q_{0}}-\sum_{q<q_{0}}\frac{1}{2^{r_{q}}(r_{q}+1)} -\sum_{q> q_{0}}\frac{1}{2^{r_{q}}} \geq r_{q_{0}}-1.
\end{align*}
Furthermore, if $k\in [1,m_{k_{r_{q_{0}}+1}}]$ and $p_{Y}^{r_{q_{0}}+1}(T_{j,k}x_{s})<1/(r_{q_{0}}+1)$ for $s<r_{q_{0}}+1,$
which holds provided that $k\geq m_{k_{r_{q_{0}}+1}+1}/(r_{q_{0}}+1)$, then we have $1+r_{q_{0}}+k\leq 1+r_{q_{0}}+m_{k_{r_{q_{0}}+1}}
\leq r_{q+1}$ due to \eqref{miruga} and therefore
\begin{align*}
p_{Y}^{q}\bigl( T_{j,k}x_{\beta}\bigr) & \leq \sum_{q\leq q_{0}}\frac{p_{Y}^{r_{q_{0}}+1}(T_{j,k}x_{r_{q}})}{2^{r_{q}}} +\sum_{q>q_{0}}\frac{p_{Y}^{r_{q_{0}}+1}(T_{j,k}x_{r_{q}})}{2^{r_{q}}}
\\ & \leq \sum_{q\leq q_{0}}\frac{1}{2^{r_{q}}(r_{q_{0}}+1)} +\sum_{q>q_{0}}\frac{p_{1+k+r_{q_{0}}}(x_{r_{q}})}{2^{r_{q}}}
\\ & \leq \frac{1}{2(r_{q_{0}}+1)}+\sum_{q>q_{0}}\frac{1}{2^{r_{q}}} \leq \frac{1}{r_{q_{0}}+1},\quad j\in {\mathbb N}_{N},
\end{align*}
which clearly implies that
\begin{align*}
d_{Y}\bigl( T_{j,k}x_{\beta},0\bigr)=
& \sum
\limits_{q=1}^{r_{q_{0}}+1}\frac{1}{2^{q}}\frac{p_{q}(T_{j,k}x_{\beta})}{1+p_{q}(T_{j,k}x_{\beta})}+\sum
\limits_{q=r_{q_{0}}+1}^{\infty}\frac{1}{2^{q}}\frac{p_{q}(T_{j,k}x_{\beta})}{1+p_{q}(T_{j,k}x_{\beta})}
\\ & \leq \frac{1}{r_{q_{0}}+1}+ \frac{1}{2^{r_{q_{0}}}},\quad j\in {\mathbb N}_{N}
\end{align*}
and $x_{\beta}$ is a $(d,span\{x_{\beta}\},m_{n},1,1)$-Li-Yorke irregular vector for $((T_{j,k})_{k\in {\mathbb N}})_{1\leq j\leq N}.$
The final statement of theorem now follows similarly as in the proofs of 
\cite[Theorem 15]{2013JFA} and \cite[Theorem 4.1]{reit}. 
\vspace{0.1cm}

Now we will  state the following corollary of Theorem \ref{na-dobro1-ly}:

\begin{cor}\label{rdctype}
Suppose that $X$ is separable, $((T_{j,k})_{k\in {\mathbb N}})_{1\leq j\leq N}$ is a sequence in $L(X,Y),$
$X_{0}$ is a dense linear subspace of $X,$ and $\lim_{k\rightarrow \infty}T_{j,k}x=0,$ $x\in X_{0},$ $j\in {\mathbb N}_{N}.$
Then the following statements are equivalent:
\begin{itemize}
\item[(i)] The tuple $((T_{j,k})_{k\in {\mathbb N}})_{1\leq j\leq N}$ is densely $(d,m_{n},1,1)$-Li-Yorke chaotic for some (all) $(m_{n})\in {\mathrm R}.$ 
\item[(ii)] The tuple $((T_{j,k})_{k\in {\mathbb N}})_{1\leq j\leq N}$ is densely $(d,m_{n},1,2)$-Li-Yorke chaotic for some (all) $(m_{n})\in {\mathrm R}$.
\item[(iii)] The tuple $((T_{j,k})_{k\in {\mathbb N}})_{1\leq j\leq N}$ is densely $(d,1,3)$-Li-Yorke chaotic.
\item[(iv)]  The tuple $((T_{j,k})_{k\in {\mathbb N}})_{1\leq j\leq N}$ is densely $(d,1,4)$-Li-Yorke chaotic.
\end{itemize}
\end{cor}

\begin{proof}
For the proof of implication (i) $\Rightarrow$ (ii), it suffices to recall that the assumption 
$\underline{d}_{m_{n}}(A)=0$ for some $(m_{n})\in {\mathrm R}$ and $A\subseteq {\mathbb N}$ implies $\underline{Bd}_{l;m_{n}}(A)=0.$ The implications  (ii) $\Rightarrow$ (iii) $\Rightarrow$ (iv) are trivial, while the implication (iv) $\Rightarrow$ (i) follows from an application of Theorem \ref{na-dobro1-ly}.
\end{proof}

In connection with Theorem \ref{na-dobro1-ly}, it should be recalled that the existence of dense Li-Yorke irregular manifolds for orbits of linear continuous operators on Banach spaces
has been analyzed in \cite[Section 4]{band}. In particular, the authors have shown that for any operator $T\in L(X),$ where $X$ is a separable Banach space, the existence of a dense linear subspace $X_{0}$ of $X$ and a strictly increasing sequence $(l_{k})$ of positive integers such that $\lim_{k\rightarrow \infty}\|T^{l_{k}}x\|=0$ for all $x\in X_{0}$ implies that the Li-Yorke chaos of $T$ is equivalent either with the existence of dense Li-Yorke irregular manifold for $T$ or the existence of an unbounded orbit (see \cite[Corollary 33]{band}). The method used in the proof of this result is substantially different from that of \cite[Theorem 15]{2013JFA} and we will not reexamine it for disjoint Li-Yorke chaos.  

Now we state the following corollary of Theorem \ref{na-dobro1-ly}, which can be deduced by using the pivot spaces $[R(C)],$ $X$ and the sequence $(({\bf T}_{j,k})_{k\in {\mathbb N}})_{1\leq j\leq N},$ where
${\bf T}_{j,k}(Cx):=T_{j,k}Cx,$ $x\in X$
for $k\in {\mathbb N}$ and $j\in {\mathbb N}_{N}:$

\begin{cor}\label{deckonja}
Suppose that $T_{j,k} : D(T_{j,k})\subseteq X \rightarrow X$ is a linear mapping, $C\in L(X)$ is an injective
mapping with dense range, as well as
\begin{align*}
R(C)\subseteq D(T_{j,k})\mbox{ and }T_{j,k}C\in L(X)\mbox{ for all }k\in {\mathbb N}\mbox{ and }j\in {\mathbb N}_{N}.
\end{align*}
Suppose, further, that $X$ is separable, $m\in {\mathbb N},$ 
$X_{0}$ is a dense linear subspace of $X,$ $(m_{n})\in {\mathrm R}$ as well as:
\begin{itemize}
\item[(i)] $\lim_{k\rightarrow \infty}T_{j,k}Cx=0,$ $x\in X_{0},$ $j\in {\mathbb N}_{N};$
\item[(ii)] there exist a vector $y\in X$ and an increasing sequence $(n_{k})$ tending to infinity such that $\lim_{k\rightarrow \infty}p_{m}(T_{j,n_{k}}Cy)=+\infty ,$ $j\in {\mathbb N}_{N}$ [$\lim_{k\rightarrow \infty}\|T_{j,n_{k}}Cy\|=+\infty ,$ $j\in {\mathbb N}_{N}$, provided that $X$ is a Banach space].
\end{itemize}
Then there exists a dense
submanifold $W$ of $X$ consisting of those vectors $x\in R(C)$ such that $x$ is $(d,m_{n})$-distributionally near to zero of type $1$ for $((T_{j,k})_{k\in {\mathbb N}})_{1\leq j\leq N}$ and for which there exists a strictly increasing subsequence $(l_{k})$ of $(n_{k})$ such that
the sequence $(p_{m}(T_{j,l_{k}}x))_{k\in {\mathbb N}}$ tends to $+\infty$ for all $j\in {\mathbb N}_{N}$ [$(\|T_{j,l_{k}}x\|)_{k\in {\mathbb N}}$ tends to $+\infty$ for all $j\in {\mathbb N}_{N},$ provided that $X$ is a Banach space]. In particular, $((T_{j,k})_{k\in {\mathbb N}})_{1\leq j\leq N}$ is densely $(d,W,1,1)$-Li-Yorke chaotic.
\end{cor}

\begin{rem}\label{obori-pjan}
Concerning possible applications of Theorem \ref{na-dobro1-ly} (similar conclusions hold for Corollary \ref{deckonja}), it should be noted the following facts with regards to the validity of condition (ii) in its formulation:
\begin{itemize}
\item[(i)] Suppose that $X$ and $Y$ are Banach spaces, $(n_{k})$ is a strictly increasing sequence and $((T_{j,k})_{k\in {\mathbb N}})_{1\leq j\leq N}$ is a sequence in $L(X,Y).$ If for each $j\in {\mathbb N}_{N}$ we have $\sum^{\infty}_{k=1}\frac{1}{\|T_{j,n_{k}}\|}<\infty ,$
then there exists $y\in X$ such that $\lim_{k\rightarrow \infty}\|T_{j,n_{k}}y\|_{Y}=\infty $ for each $j\in {\mathbb N}_{N}.$
\item[(ii)]  Suppose that $X$ is a complex Hilbert space, $Y$ is a complex Banach space, $(n_{k})$ is a strictly increasing sequence and $((T_{j,k})_{k\in {\mathbb N}})_{1\leq j\leq N}$ is a sequence in $L(X,Y).$ If for each $j\in {\mathbb N}_{N}$ 
we have $\sum^{\infty}_{k=1}\frac{1}{\|T_{j,n_{k}}\|^{2}}<\infty ,$ then there exists $y\in X$ such that $\lim_{k\rightarrow \infty}\|T_{j,n_{k}}y\|_{Y}=\infty $ for each $j\in {\mathbb N}_{N}.$
\end{itemize}
The statements (i) and (ii) are trivial consequences of \cite[Proposition 3.9]{ddc}, which slightly extend one of the main results of article \cite{milervrs} by V. M\"uller and J. Vr\v sovsk\'y.
\end{rem}

We close this section with the observation that \cite[Proposition 3]{cheli} can be reformulated for disjointness, showing that some investigations can be reduced to the case in which $\tilde{X}=X.$

\subsection{Applications to shift operators}\label{shiftojed}

Suppose that $X$ is a Fr\' echet sequence space in which $(e_{n})_{n\in {\mathbb N}}$ is a basis (see e.g. \cite[Section 4.1]{erdper}). In this subsection, we will always assume that for each $j\in {\mathbb N}_{N}$ the unilateral weighted backward shift $T_{j}$ on $X$ is
given by
\begin{align*}
T_{j}& \bigl\langle x_{n}\bigr\rangle_{n\in {\mathbb N}}:=\bigl\langle w_{j,n}x_{n+1}\bigr\rangle_{n\in {\mathbb N}},\quad \bigl\langle x_{n}\bigr\rangle_{n\in {\mathbb N}}\in X,\mbox{ and }
\\ & D\bigl(T_{j}\bigr):=\Bigl\{ \bigl\langle x_{n}\bigr\rangle_{n\in {\mathbb N}} \in X : T_{j}\bigl\langle x_{n}\bigr\rangle_{n\in {\mathbb N}}\in X\Bigr\} \quad \bigl( j\in {\mathbb N}_{N}\bigr).
\end{align*}
The continuity of operators $T_{j}$ will not be assumed a priori. 

We start by providing the following illustrative example:

\begin{example}\label{prika-shift}
Suppose that $X:=l^{1}({\mathbb N}),$ 
$0<\zeta_{1}\leq \zeta_{2}\leq \cdot \cdot \cdot \leq \zeta_{n}\leq 1,$
$\langle \omega_{n}\rangle_{n\in {\mathbb N}}:= \langle \frac{2n}{2n-1}\rangle_{n\in {\mathbb N}}$ 
and $\langle \omega_{j,n}\rangle_{n\in {\mathbb N}}:= \langle (\frac{2n}{2n-1})^{\zeta_{j}}\rangle_{n\in {\mathbb N}}$ 
for all $j\in {\mathbb N}_{N};$ see also \cite[Theorem 3.5]{turkish-notes} and \cite[Example 4.9]{reit}.
Then for each $j\in {\mathbb N}_{N}$ the corresponding 
operator $T_{j}$ is topologically mixing, absolutely Ces\`aro bounded and therefore not distributionally chaotic; albeit this basically follows from the argumentation used in the proof of \cite[Theorem 3.5]{turkish-notes}, we will include all relevant details for the sake of completeness.
Applying Stirling's formula, we get that 
\begin{align}\label{bavcxz}
\beta (n):=\prod_{j=1}^{n}\omega_{i} \sim \sqrt{\pi n},\quad n\rightarrow +\infty.
\end{align}
Using this and \cite[Proposition 3.1]{turkish-notes}, we get that the operator $T_{j}$ is topologically mixing ($j\in {\mathbb N}_{N}$). Furthermore, for each
$n\in {\mathbb N},$ $X\ni x=\langle x_{k}\bigr\rangle_{k\in {\mathbb N}} \neq 0$ and $j\in {\mathbb N}_{N}$
we have:
\begin{align*}
\frac{1}{n}&\sum_{l=1}^{n}\bigl\|T_{j}^{l}x\bigr\|=\frac{1}{n}\sum_{l=1}^{n}\sum_{k=l+1}^{\infty}\bigl( \omega_{k-l}\cdot \cdot \cdot \omega_{k-1} \bigr)^{\zeta_{j}}|x_{k}|
\\ & =\frac{1}{n}\sum_{k=2}^{\infty}\sum_{l=1}^{\min(k-1,n)}\bigl( \omega_{k-l}\cdot \cdot \cdot \omega_{k-1} \bigr)^{\zeta_{j}}|x_{k}|
\\ & =\frac{1}{n}\sum_{k=2}^{n+1}\sum_{l=1}^{k-1}\bigl( \omega_{k-l}\cdot \cdot \cdot \omega_{k-1} \bigr)^{\zeta_{j}}|x_{k}|+\frac{1}{n}\sum_{k=n+2}^{\infty}\sum_{l=1}^{n}\bigl( \omega_{k-l}\cdot \cdot \cdot \omega_{k-1} \bigr)^{\zeta_{j}}|x_{k}|
\\ & \leq \frac{1}{n}\sum_{k=2}^{n+1}\sum_{l=1}^{k-1}\omega_{k-l}\cdot \cdot \cdot \omega_{k-1} |x_{k}|+\frac{1}{n}\sum_{k=n+2}^{\infty}\sum_{l=1}^{n} \omega_{k-l}\cdot \cdot \cdot \omega_{k-1}|x_{k}|.
\end{align*}
For the estimation of second addend, the arguments used in \cite{turkish-notes} show that it does not exceed $2\|x\|.$ For the first addend, we can employ \eqref{bavcxz}
 in order to see that there exist two finite constants $c>0$ and $c_{1}>0$ such that
\begin{align*}
\frac{1}{n}\sum_{k=2}^{n+1}\sum_{l=1}^{k-1}\omega_{k-l}\cdot \cdot \cdot \omega_{k-1} |x_{k}|\leq \frac{c}{n}\sum_{k=2}^{n+1}|x_{k}|\sum_{l=1}^{k-1}\sqrt{\frac{k}{k-l}}
\leq \frac{c_{1}}{n}\sum_{k=2}^{n+1}|x_{k}|\leq c_{1}\|x\|,
\end{align*}
finishing the proof of fact that $T_{j}$ is absolutely Ces\`aro bounded and consequently not distributionally chaotic ($j\in {\mathbb N}_{N}$). This implies that the operators $T_{1},\cdot \cdot \cdot,T_{N}$ cannot be $(d,X,i)$-distributionally chaotic for any $i\in {\mathbb N}_{8};$ see \cite{ddc}. On the other hand, the operator $T_{1}$ is clearly Li-Yorke chaotic and possesses a Li-Yorke irregular vector $y.$ In our concrete example, we have $\|T_{1}^{n_{k}}x\|\leq \|T_{j}^{n_{k}}x\|$ for any $x\in X,$ $j\in {\mathbb N}_{N}\setminus \{1\}$ and any strictly increasing sequence $(n_{k})$ so that
Corollary \ref{deckonja} with $C=I$ yields that the operators $T_{1},\cdot \cdot \cdot,T_{N}$ are densely $(d,1,1)$-Li-Yorke chaotic. Finally, we want to note that these operators cannot be $d$-hypercyclic due to \cite[Theorem 2.1]{prcko-kolekt} (cf. also \cite{bp07} and \cite{research} for basic results given in this direction).
\end{example}

Now we will provide an application of Corollary \ref{deckonja} to unbounded unilateral backward shift operators:

\begin{example}\label{malo-zajebano}
Let $S:=\{n_{k}: k\in {\mathbb N}\},$ where $(n_{k})$ is a strictly increasing sequence of positive integers, and let the operator $A_{j} \langle x_{n}\rangle_{n\in {\mathbb N}}:=\langle (1+j)^{n+1}x_{n+1}\rangle_{n\in {\mathbb N}}$ act with its maximal domain in the space $X:=c_{0}({\mathbb N})$ for $j\in {\mathbb N}_{N}$. Set
$C\langle x_{n}\rangle_{n\in {\mathbb N}}:=\langle (3/2)^{-n^{2}}x_{n}\rangle_{n\in {\mathbb N}},$ $\langle x_{n}\rangle_{n\in {\mathbb N}}\in X.$ Then it is clear that $C\in L(X)$ is injective and $R(C)$ is dense in $X.$
Furthermore, it can be easily seen that $A_{j}^{k}C\in L(X)$ for all $j\in {\mathbb N}_{N}$ and $k\in {\mathbb N}$; strictly speaking, for any vector $x:=\langle x_{n}\rangle_{n\in {\mathbb N}}$ in $X$
we have
\begin{align*}
\bigl\|A_{j}^{k}Cx\bigr\|& \leq \|x\| \sup_{l\geq 1}(1+j)^{lk+\frac{k(k+1)}{2}}(3/2)^{-(l+k)^{2}}
\\ & \leq \|x\| (1+j)^{\frac{k(k+1)}{2}}\sup_{l\geq 1}\Bigl((1+j)^{k}\Bigr)^{l}(3/2)^{-l^{2}}
\\ & \leq c_{j,k}\|x\|,
\end{align*}
for some positive finite constant $c_{j,k}>0.$
On the other hand, with sequence 
$x:=\langle 1/n\rangle_{n\in {\mathbb N}}$
we have
\begin{align*}
\bigl\|A_{j}^{k}Cx\bigr\|&= \sup_{l\geq 1}(1+j)^{lk+\frac{k(k+1)}{2}}(3/2)^{-(l+k)^{2}}\bigl|x_{k+l}\bigr|
\\ & \geq (1+j)^{k^{2}+\frac{k(k+1)}{2}}(3/2)^{-(2k)^{2}}\bigl|x_{2k}\bigr|
\\ & =(1+j)^{k^{2}+\frac{k(k+1)}{2}}(3/2)^{-(2k)^{2}}/2k\rightarrow \infty,\quad k\in {\mathbb N}.
\end{align*}
Define now
$T_{j,k}:=A_{j}^{k}$ if $k\in S$ and $T_{j,k}:=(1+\|A_{j}^{k}C\|)^{-3}A_{j}^{k}$ if $k\notin S.$ By the above argumentation, we have that the requirements of Corollary \ref{deckonja} are satisfied, so that $((T_{j,k})_{k\in {\mathbb N}})_{1\leq j\leq N}$ is densely $(d,X,1,1)$-Li-Yorke chaotic.
\end{example}

Arguing as in \cite[Example 5.3]{ddc}, we can prove that there exist two distributionally chaotic unilateral backward weighted shifts on the space $X:=c_{0}({\mathbb N})$ which cannot be $(d,X,n,1,i)$-Li-Yorke chaotic for $i\in {\mathbb N}_{2}$ or 
$(d,X,n,3,i)$-Li-Yorke chaotic for $i\in {\mathbb N}_{2}.$ 

Further on, 
it is clear that Theorem \ref{na-dobro1-ly} and Corollary \ref{deckonja} cannot be applied in the analysis of weighted forward shifts in Fr\' echet sequence spaces. 
On the other hand, we can prove directly that tuples of such operators are disjoint Li-Yorke chaotic with $e_{1}$ being the corresponding disjoint Li-Yorke irregular vector:

\begin{example}\label{count-grof-ly32445} 
Let
a weighted forward shift $F_{\omega} \in L(
l^{2})$ be defined by $F_{\omega} (x_{1}, x_{2}, \cdot \cdot \cdot) \mapsto (0, \omega_{1}x_{1},\omega_{2}x_{2},\cdot \cdot \cdot),$ where $\omega =(\omega_{k})_{k\in {\mathbb N}}$ consists of sufficiently large blocks of $2$'s and blocks of $(1/2)$'s.
In a great number of concrete situations, we have that the operators $c_{j}F_{\omega},$ where $c_{j}\in {\mathbb K}\setminus \{0\}$ for $j\in {\mathbb N}_{N},$ are $(d,span\{e_{1}\},1,1)$-Li-Yorke chaotic. Observe, finally, that these operators cannot be 
disjoint hypercyclic because $F_{\omega}$ and its multiples cannot be hypercyclic.   
\end{example}

\section{Applications to abstract PDEs in Fr\'echet spaces}\label{well-posed}

The main aim of this section is to continue the research raised in \cite{nis-dragan} concerning Li-Yorke chaotic solutions of abstract PDEs of first order. In contrast with the above-mentioned article, we consider here Li-Yorke chaotic solutions of abstract fractional PDEs as well. For the sake of brevity, we will consider only continuous counterpart of disjoint $(\tilde{X},m_{n},1,1)$-Li-Yorke chaos here, which will be called disjoint $(\tilde{X},f,1,1)$-Li-Yorke chaos (cf. \cite{reit} and \cite{ddcNS} for the notion of disjoint reiterative $\tilde{X}_{f}$-distributional chaos
and certain applications to abstract PDEs).

Suppose that $T(t) : D(T(t)) \subseteq X \rightarrow Y$ is a linear possibly not continuous mapping ($t\geq 0$). By $Z(T)$ we denote the set consisting of those vectors $x\in X$ such that
$x\in D(T(t))$ for all $t\geq 0$ as well as that the mapping $t\mapsto T(t)x,$ $t\geq 0$ is continuous.  
Denote by $m(\cdot)$ the Lebesgue measure on $[0,\infty)$ and by ${\mathrm F}$ the class consisting of all increasing mappings $f : [0,\infty) \rightarrow [1,\infty)$ satisfying that  $\liminf_{t\rightarrow +\infty}\frac{f(t)}{t}>0.$

We will use the following continuous counterpart of Definition \ref{prckojed}:

\begin{defn}\label{prckojed-prim} (\cite{F-operatori})
Let $A\subseteq [0,\infty)$, and let $f \in {\mathrm F}.$ Then the lower $f$-density of $A,$ denoted by $\underline{d}_{f}(A),$ is defined through:
$$
\underline{d}_{f} (A):=\liminf_{t\rightarrow \infty}\frac{m(A \cap [0,f(t)])}{t}.
$$
\end{defn}

Consider the following condition: 
\begin{align}\label{frechet-banach1-kontinuiran}
\begin{split}
& (\exists m\in {\mathbb N})(\forall k\in {\mathbb N})(\exists t_{k}\in [0,\infty)) \mbox{  s.t. }t_{k}<t_{k+1},\mbox{ }\lim_{k\rightarrow \infty}t_{k}=+\infty,\ k\in {\mathbb N},
\\ &\mbox{ and }\lim_{k\rightarrow \infty}p_{m}^{Y}\bigl(x_{j,t_{k}}-y_{j,t_{k}}\bigr)=+\infty, \ k\in {\mathbb N},\ j\in {\mathbb N}_{N},\mbox{ provided that }Y\mbox{ is a Fr\' echet space}, or
\\
& (\forall k\in {\mathbb N})(\exists t_{k}\in [0,\infty)) \mbox{  s.t. }t_{k}<t_{k+1},\mbox{  }\lim_{k\rightarrow \infty}t_{k}=+\infty,\ k\in {\mathbb N},
\\ & \mbox{ and }\lim_{k\rightarrow \infty}\bigl\|x_{j,t_{k}}-y_{j,t_{k}}\bigr\|=+\infty, \  j\in {\mathbb N}_{N},\mbox{ provided that }Y\mbox{ is a Banach space}.
\end{split}
\end{align}

\begin{defn}\label{LY-unbounded-fric-cont} 
Suppose that $\tilde{X}$ is a linear subspace of $X,$ $T_{j}(t) : D(T_{j}(t)) \subseteq X \rightarrow Y$ is a linear possibly not continuous mapping ($t\geq 0,$ $j\in {\mathbb N}_{N}$) and $f \in {\mathrm F}.$ If there exist an uncountable
set $S\subseteq \bigcap_{j\in {\mathbb N}_{N}}Z(T_{j}) \cap \tilde{X}$ and $m\in {\mathbb N},$ in the case that $Y$ is a Fr\' echet space, such that \eqref{frechet-banach1-kontinuiran} holds and for each $\epsilon>0$ and for each pair $x,\
y\in S$ of distinct points we have that 
\begin{align}\label{priqaz}
\underline{d}_{f}\Biggl( \bigcup_{j\in {\mathbb N}_{N}} \bigl\{t\geq 0 : d_{Y}\bigl(T_{j}(t)x,T_{j}(t)y\bigr)
\geq \epsilon \bigr\}\Biggr)=0,
\end{align}
then we say that the tuple $((T_{j}(t))_{t\geq 0})_{1\leq j\leq N}$ is 
$(d,\tilde{X}, f,1 ,1)$-Li-Yorke chaotic ($(d,f,1,1)$-Li-Yorke chaotic, if $\tilde{X}=X$). 
Furthermore, we say that  the tuple $((T_{j}(t))_{t\geq 0})_{1\leq j\leq N}$ is densely
$(d,\tilde{X}, f,1 ,1)$-Li-Yorke chaotic iff $S$ can be chosen to be dense in $\tilde{X}.$
The set $S$ is said to be $(d,{\sigma_{\tilde{X}}}_{f})$-Li-Yorke scrambled set ($(d,\sigma_{f})$-scrambled set in the case that $\tilde{X}=X$)     
of $((T_{j}(t))_{t\geq 0})_{1\leq j\leq N}.$

If $q\geq 1$ and $f(t):=1+t^{q}$ ($t\geq 0$), then we particularly obtain the notions of (dense) disjoint $\tilde{X}_{q}$-Li-Yorke chaos, (dense) disjoint $q$-Li-Yorke chaos, $(d,{\sigma_{\tilde{X}}}_{q})$-Li-Yorke scrambled set and $(d,\sigma_{q})$-Li-Yorke scrambled set for $((T_{j}(t))_{t\geq 0})_{1\leq j\leq N}.$
\end{defn}

The main result for applications is the following continuous counterpart of Theorem \ref{na-dobro1-ly}; the proof can be deduced similarly and therefore omitted (cf. \cite{mendoza} and \cite{reit} for more details):

\begin{thm}\label{na-dobro1-ly-cont}
Suppose that $X$ is separable, $m\in {\mathbb N},$ $f \in {\mathrm F},$ $((T_{j}(t))_{t\geq 0})_{1\leq j\leq N}$ is a sequence of strongly continuous operator families in $L(X,Y),$
$X_{0}$ is a dense linear subspace of $X,$ as well as:
\begin{itemize}
\item[(i)] $\lim_{t\rightarrow \infty}T_{j}(t)x=0,$ $x\in X_{0},$ $j\in {\mathbb N}_{N};$
\item[(ii)] there exist a vector $y\in X$ and an increasing sequence $(t_{k}')$ tending to infinity such that $\lim_{k\rightarrow \infty}p_{m}^{Y}(T_{j}(t_{k}')y)=+\infty ,$ $j\in {\mathbb N}_{N}$ [$\lim_{k\rightarrow \infty}\|T_{j}(t_{k}')y\|_{Y}=+\infty ,$ $j\in {\mathbb N}_{N},$ provided that $Y$ is a Banach space].
\end{itemize}
Then there exist a dense
submanifold $W$ of $X$ consisting of those vectors $x$ which are disjoint $f$-distributionally near to zero for $((T_{j}(t))_{t\geq 0})_{1\leq j\leq N},$ in the sense that for each number $\epsilon>0$ we have that \eqref{priqaz} holds with $y=0,$ and for which there exists a strictly increasing subsequence $(t_{k})$ of  $(t_{k}')$ tending to infinity such that
the sequence $(p_{m}(T_{j}(t_{k})x))_{k\in {\mathbb N}}$ tends to $+\infty$ for all $j\in {\mathbb N}_{N}$ [$(\|T_{j}(t_{k})x\|_{Y})_{k\in {\mathbb N}}$ tends to $+\infty$ for all $j\in {\mathbb N}_{N},$ provided that $Y$ is a Banach space]. In particular, the tuple $((T_{j}(t))_{t\geq 0})_{1\leq j\leq N}$ is densely $(d,W,f,1,1)$-Li-Yorke chaotic.
\end{thm}

We continue by providing two simple remarks:

\begin{rem}\label{maliremark}
Suppose that $X$ and $Y$ are Banach spaces as well as that the tuple $((T_{j}(t))_{t\geq 0})_{1\leq j\leq N}$ of strongly continuous operator families in $L(X,Y)$ satisfies (i) and 
$$
\lim_{t\rightarrow \infty}\sum_{j=1}^{N}\bigl\|T_{j}(t)\bigr\|_{L(X,Y)}=+\infty.
$$ 
Considering the operators $T(t) : X^{N} \rightarrow Y^{N}$ defined by  $T(t)(x_{1},\cdot \cdot \cdot,x_{N}):=(T_{1}(t)x_{1},\cdot \cdot \cdot,T_{N}(t)x_{N})$ for $t\geq 0$ and $x_{1},\cdot \cdot \cdot,x_{N}\in X,$ it can be simply proved that 
there exist a strictly increasing sequence $(t_{k})$ of positive real numbers and a vector $(x_{1},\cdot \cdot \cdot,x_{N}) \in X^{N}$ such that $\lim_{k\rightarrow \infty}[\|T_{1}(t_{k})x_{1}\|_{Y}+\cdot \cdot \cdot +\|T_{N}(t_{k})x_{N}\|_{Y} ]=+\infty.$ But, this does not imply the validity of condition (ii)
in Theorem \ref{na-dobro1-ly-cont}.
\end{rem}

\begin{rem}\label{dist-chaos-remarkrweq}
Suppose that $X$ is a Banach space and $T_{j}\in L(X)$ for all $j\in {\mathbb N}_{N}.$ If there exists an element $y\in X$ such that $\lim_{k\rightarrow \infty}\|T_{j}^{n_{k}}y\|=+\infty$ for all $j\in {\mathbb N}_{N},$ then for each integer $m\in {\mathbb N}$ we have that $\lim_{k\rightarrow \infty}\|T_{j}^{m\lfloor n_{k}/m\rfloor}y\|=+\infty$ for all $j\in {\mathbb N}_{N},$ as well; this can be deduced along the lines of proof of \cite[Proposition 2.4]{parij}. Similarly, if $X$ is a Banach space, $(T_{j}(t))_{t\geq 0}$ is a strongly continuous semigroup on $X$ for each $j\in {\mathbb N}_{N},$ $y\in X$ and $\lim_{k\rightarrow \infty}\|T_{j}(t_{k}')y\|=+\infty,$ $j\in {\mathbb N}_{N}$ for some strictly increasing seqeunce $(t_{k}')$ tending to infinity, then for each $t_{0}>0$ we have $\lim_{k\rightarrow \infty}\|T_{j}(t_{0}\lfloor t_{k}'/t_{0}\rfloor)y\|=+\infty,$ $j\in {\mathbb N}_{N}.$
\end{rem}

The trivial case in which the requirements of Theorem \ref{na-dobro1-ly-cont} hold, and which can be also reworded for disjoint $m_{n}$-distributional chaos, is given as follows:

\begin{example}\label{prcko-frcko}
Suppose that $X$ is separable, $m\in {\mathbb N},$ $f \in {\mathrm F},$ $(T_{1}(t))_{t\geq 0}$ is a strongly continuous operator family in $L(X,Y)$ satisfying the condition (ii) of  Theorem \ref{na-dobro1-ly-cont} with $j=1,$
$X_{0}$ is a dense linear subspace of $X,$ as well as
for each $x\in X_{0}$ there exists a finite number $t_{0}>0$ such that $T_{1}(t)x=0$ for all $t>t_{0}.$ 
Suppose, further, that for each integer
$j\in {\mathbb N}_{N} \setminus \{1\}$ we have that $f_{j} :[0,\infty) \rightarrow {\mathbb K}$ is a given continuous function as well as that there exists a sufficiently small number $c>0$ such that
$|f_{j}(t)|\geq c,$ $t\geq 0,$ $j\in {\mathbb N}_{N} \setminus \{1\}.$ Define $T_{j}(t):=f_{j}(t)T_{1}(t),$ $t\geq 0,$ $j\in {\mathbb N}_{N} \setminus \{1\}.$ Then it can be simply checked that all requirements of Theorem \ref{na-dobro1-ly-cont} hold.
In particular, if $d/dx$ is the infinitesimal generator of a Li-Yorke chaotic strongly continuous translation semigroup in the space $L_{\rho}^{p}([0,\infty))$ of $C_{0,\rho}([0,\infty)),$ then the strongly continuous semigroups generated by the operators
\begin{align}\label{idiotishen-jedan}
\frac{d}{dx},\, \frac{d}{dx}+\omega_{2},\cdot \cdot \cdot, \,\frac{d}{dx} +\omega_{N}
\end{align}
are densely $(d,X,f,1,1)$-Li-Yorke chaotic, where $\omega_{2},\cdot \cdot \cdot,\omega_{N}$ are certain scalars from the field ${\mathbb K}$ having non-negative real parts (cf. \cite[Lemma 4.6]{fund} for precise definition of generator); the same statement holds
for Li-Yorke chaotic strongly continuous semigroups induced by semiflows for which the condition \cite[(D), p. 25]{nis-dragan} holds. Observe, finally, that there exists a strongly continuous translation semigroup $(T_{1}(t))_{t\geq 0}$ on the space $L_{\rho}^{p}([0,\infty)),$ with a certain weight function $\rho :[0,\infty) \rightarrow (0,\infty),$ which is topologically mixing but not distributionally chaotic (see e.g. \cite[Example 4.2]{bara1} and \cite[Theorem 2.1]{gimenez}). This implies that the strongly continuous semigroups $(T_{1}(t))_{t\geq 0},(e^{\omega_{2}t}T_{1}(t))_{t\geq 0},\cdot \cdot \cdot,(e^{\omega_{N}t}T_{1}(t))_{t\geq 0}$, whose generators are given by \eqref{idiotishen-jedan}, cannot be $(d,X,i)$-distributionally chaotic for any $i\in {\mathbb N}_{8};$ see \cite{nsjom-novi} for the notion.
\end{example}

The following corollary of Theorem \ref{na-dobro1-ly-cont} can be deduced similarly as in discrete case: 

\begin{cor}\label{deckonja-frcj}
Suppose that $X$ is separable, $m\in {\mathbb N},$ $f \in {\mathrm F},$ $((T_{j}(t))_{t\geq 0})_{1\leq j\leq N}$ is a sequence of linear operator families in $L(X,Y),$
$X_{0}$ is a dense linear subspace of $X,$ $C\in L(X)$ is injective with the mapping $t\mapsto T_{j}(t)Cx$ be well-defined and continuous for $t\geq 0$ and $j\in {\mathbb N}_{N},$ as well as:
\begin{itemize}
\item[(i)] $\lim_{t\rightarrow \infty}T_{j}(t)Cx=0,$ $x\in X_{0},$ $j\in {\mathbb N}_{N};$
\item[(ii)] there exist a vector $y\in X$ and an increasing sequence $(t_{k}')$ tending to infinity such that $\lim_{k\rightarrow \infty}p_{m}^{Y}(T_{j}(t_{k}')Cy)=+\infty ,$ $j\in {\mathbb N}_{N}$ [$\lim_{k\rightarrow \infty}\|T_{j}(t_{k}')Cy\|_{Y}=+\infty ,$ $j\in {\mathbb N}_{N},$ provided that $Y$ is a Banach space].
\end{itemize}
Then there exist a dense
submanifold $W$ of $X$ consisting of those vectors $x\in R(C)$ which are disjoint $f$-distributionally near to zero for $((T_{j}(t))_{t\geq 0})_{1\leq j\leq N},$ in the sense that for each number $\epsilon>0$ we have that \eqref{priqaz} holds with $y=0,$ and for which there exists a strictly increasing subsequence $(t_{k})$ of  $(t_{k}')$ tending to infinity such that
the sequence $(p_{m}(T_{j}(t_{k})x))_{k\in {\mathbb N}}$ tends to $+\infty$ for all $j\in {\mathbb N}_{N}$ [$(\|T_{j}(t_{k})x\|_{Y})_{k\in {\mathbb N}}$ tends to $+\infty$ for all $j\in {\mathbb N}_{N},$ provided that $Y$ is a Banach space]. In particular, the tuple $((T_{j}(t))_{t\geq 0})_{1\leq j\leq N}$ is densely $(d,W,f,1,1)$-Li-Yorke chaotic.
\end{cor}

\begin{rem}\label{dist-chaos-remark}
The condition (ii) in Theorem \ref{na-dobro1-ly-cont} (Corollary \ref{deckonja-frcj}) is satisfied in many concrete cases in which there exists a vector $y\in X$ such that $\lim_{t\rightarrow \infty}p_{m}^{Y}(T_{j}(t)y)=+\infty ,$ $j\in {\mathbb N}_{N}$ ($\lim_{t\rightarrow \infty}p_{m}^{Y}(T_{j}(t)Cy)=+\infty ,$ $j\in {\mathbb N}_{N}$); but, in this case, for each function $f \in {\mathrm F}$ the tuple $((T_{j}(t))_{t\geq 0})_{1\leq j\leq N}$ will be densely $(d,f)$-distributionally chaotic (cf. \cite{ddcNS} for discrete case) in the sense 
that there exist an uncountable
set $S\subseteq X$ and a finite number $\sigma>0$ such that 
\begin{align*}
\underline{d}_{f}\Biggl( \bigcup_{j\in {\mathbb N}_{N}} \bigl\{t\geq 0 : d_{Y}\bigl(T_{j}(t)x,T_{j}(t)y\bigr)
\leq \sigma \bigr\}\Biggr)=0
\end{align*} 
and for each $\epsilon>0$ and for each pair $x,\
y\in S$ of distinct points we have that 
\begin{align*}
\underline{d}_{f}\Biggl( \bigcup_{j\in {\mathbb N}_{N}} \bigl\{t\geq 0 : d_{Y}\bigl(T_{j}(t)x,T_{j}(t)y\bigr)
\geq \epsilon \bigr\}\Biggr)=0,
\end{align*}
which is a much stronger notion than that of dense $(d,X,1,1)$-Li-Yorke chaos (furthermore, for each number $t_{0}>0$ the operators $T_{1}(t_{0}),\cdot \cdot \cdot, T_{N}(t_{0})$ will be densely $(d,m_{n})$-distributionally chaotic, where $m_{n}:=\lceil f(n)\rceil$). A concrete example for abstract fractional PDEs can be simply constructed: Suppose that $0<\alpha<2$ and for each $j\in {\mathbb N}_{N}$ we have that $(T_{j}(t))_{t\geq 0}\subseteq L(X)$
is an $\alpha$-times $C$-regularized resolvent family with the integral generator $A_{j},$ as well as that $R(C)$ is dense in $X$ and there exist a vector $x\in X$ and a number $\lambda_{j}\in \Sigma_{\alpha \pi/2}$ such that $A_{j}x=\lambda_{j}x;$ see \cite{knjigaho} for the notion. Then, due to \cite[Lemma 3.3.1]{knjigaho}, for each $j\in {\mathbb N}_{N}$ we have that $T_{j}(t)x=E_{\alpha}(t^{\alpha}\lambda_{j})Cx,$ $t\geq 0,$ so that the asymptotic expansion formulae for the Mittag-Leffler functions \cite{knjigaho} yield that 
the tuple $((T_{j}(t))_{t\geq 0})_{1\leq j\leq N}$ is densely $(d,f)$-distributionally chaotic for each function $f \in {\mathrm F}.$ 
\end{rem}

We will provide the following illustrative application of Corollary \ref{deckonja-frcj}; see also \cite[Example 5.12]{mendoza} and \cite[Example 3.2.39]{knjigaho}:

\begin{example}\label{jebiga-radu}
Suppose that $n\in {\mathbb N},$ $f \in {\mathrm F},$ 
$\rho(t):=\frac{1}{t^{2n}+1},\ t\in {\mathbb R},$ $Af:=f^{\prime},$
$D(A):=\{f\in C_{0,\rho}({\mathbb R}) : f^{\prime} \in
C_{0,\rho}({\mathbb R})\},$ $E_{n}:=(C_{0,\rho}({\mathbb
R}))^{n+1},$ $D(A_{n}):=D(A)^{n+1}$ and $A_{n}(f_{1},\cdot \cdot
\cdot ,f_{n+1}):=(Af_{1}+Af_{2},Af_{2}+Af_{3},\cdot \cdot \cdot ,
Af_{n}+Af_{n+1},Af_{n+1}),$ $(f_{1},\cdot \cdot \cdot, f_{n+1}) \in
D(A_{n}).$ Then $\pm A_{n}$
generate global polynomially bounded $n$-times integrated semigroups
$(S_{n,\pm}(t))_{t\geq 0},$ neither $A_{n}$ nor $-A_{n}$
generates a local $(n-1)$-times integrated semigroup and we have that, for every $\varphi_{1},...,
\varphi_{n+1} \in {\mathcal D},$
$$
G_{\pm,n}(\delta_{t})\bigl(\varphi_{1}, ... ,
\varphi_{n+1}\bigr)^{T}=\bigl(\psi_{1}, ... , \psi_{n+1}\bigr)^{T},
$$
where $G_{\pm,n}$ denote distribution semigroups generated by $\pm A_{n},$ and
$$
\psi_{i}(\cdot)=\sum \limits_{j=0}^{n+1-i}\frac{(\pm
t)^{j}}{j!}\varphi_{i+j}^{(j)}(\cdot \pm t),\quad 1\leq i\leq n+1.
$$  
Denote by ${\mathbf G}_{n}$ the corresponding distribution cosine function generated by $A_{n}^{2}.$ Using Corollary 
\ref{deckonja-frcj}, the first observation from Remark \ref{dist-chaos-remark}, and
arguing as in the above-mentioned examples, we can prove that the operator families $((e^{ia_{j}t}(1+t)^{b_{j}}G_{\pm,n}(\delta_{t})_{t\geq 0})_{1\leq j\leq N}$ and $((e^{ia_{j}t}(1+t)^{b_{j}}{\mathbf G}_{n}(\delta_{t}))_{t\geq 0})_{1\leq j\leq N}$ are densely\\ $(d,f)$-distributionally chaotic, provided that $a_{j}\in {\mathbb R}$ and $b_{j}\in [0,n+1)$ for all $j\in {\mathbb N}_{N}.$ The interested reader may simply write down the corresponding abstract Cauchy problems of first and second order which do have such operator families as solutions.
\end{example}

Concerning possible applications to the abstract ill-posed Cauchy problems of first and second order, we should also mention the paper 
\cite{ckpv}:
\begin{itemize}
\item[(i)] (see \cite[Example 2.12]{ckpv}) Consider the general situation of this example with the first inclusion in the equation \cite[(2.7)]{ckpv} being satisfied. If there exists a complex number $\lambda \in \Omega$ such that $\Re(P_{j}(-\lambda))>0$ for all $j\in {\mathbb N}_{N},$ then Corollary \ref{deckonja-frcj} is applicable to the entire $C$-regularized groups $((T_{j}(z))_{z\in {\mathbb C}})_{1\leq j\leq N};$ the tuple  $((C^{-1}T_{j}(t))_{t\geq 0})_{1\leq j\leq N}$ will be densely $(d,f)$-distributionally chaotic for any $f\in {\mathrm F}.$
\item[(ii)]  (see \cite[Example 2.13]{ckpv}) Similarly, we can simply modify the correpsonding conditions used in this example to conclude that the operators $A_{j}$ will generate exponentially equicontinuous, analytic $\zeta$-times integrated semigroups $(S_{\zeta}^{j}(t))_{t\geq 0}$ of angle $\pi/2$ on the product space $X\equiv E\times E$ ($j\in {\mathbb N}_{N}$). 
Then we can apply Corollary \ref{deckonja-frcj} in order to show that the tuple  $((C^{-1}S_{\zeta}^{j}(t))_{t\geq 0})_{1\leq j\leq N}$ is densely $(d,f)$-distributionally chaotic for any $f\in {\mathrm F},$ with a certain regularizing operator $C\in L(X).$   
\end{itemize}

Fairly complete analysis of disjoint Li-Yorke chaos for strongly continuous semigroups induced by semiflows is without scope of this paper (cf. \cite{kalmes} and \cite{nis-dragan} for related problematic).
We will only revisit \cite[Example 3.19]{kalmes} and \cite[Example 3.1.41(iii)]{knjigaho} to close the whole paper:

\begin{example}\label{ns-faul}
Suppose that any element of a real matrix $[a_{js}]_{1\leq j\leq N,1\leq s\leq m}$ is a
positive real number and, for every $j,\ s\in {{\mathbb N}_{N}}$
with $j\neq s,$ there exists an index $l\in {{\mathbb N}_{m}}$ such that $a_{jl}\neq a_{sl}.$ Let $\Omega:={{\mathbb R}^{m}}$ and
$q>\frac{m}{2}.$
Define semiflows $\varphi_{j} : [0,\infty) \times \Omega \rightarrow
\Omega,$ $j=1,2,\cdot \cdot \cdot , N$ and $\rho : \Omega
\rightarrow (0,\infty)$ as follows:
\begin{align*}
\varphi_{j}(t,x_{1},\cdot \cdot \cdot
,x_{m}):=\Bigl(e^{a_{j1}t}x_{1},\cdot \cdot \cdot , e^{a_{jm}t}x_{m}\Bigr)
\mbox{ and}
\end{align*}
\begin{align*}
\rho(x_{1},\cdot \cdot \cdot
,x_{m}):=\frac{1}{(1+|x|^{2})^{q}},\ t\geq 0,\ x=(x_{1},\cdot
\cdot \cdot ,x_{m})\in \Omega.
\end{align*}
Define $[T_{\varphi_{j}}(t)f](x):=f(\varphi_{j}(t,x)),$ $t\geq 0,$ $x\in \Omega ,$ $j\in {\mathbb N}_{N}.$ Then, for every $j\in {\mathbb N}_{N},$
$(T_{\varphi_{j}}(t))_{t\geq 0}$ is a non-hypercyclic
strongly continuous semigroup in 
$C_{0,\rho}(\Omega).$ 
Suppose now that $0<\epsilon_{1}<\epsilon_{2}<\cdot \cdot \cdot<\epsilon_{N}<\min\{a_{js}: 1\leq j\leq N,1\leq s\leq m\}.$
Since for each $j\in {\mathbb N}_{N}$ and $f\in {\mathcal D}({\mathbb R}^{m})$ (the space of scalar-valued smooth test functions with compact support and domain ${\mathbb R}^{m}$) we have $\|T_{\varphi_{j}}(t)f\|\leq \|f\|_{\infty},$ it can be simply seen that the condition (i) of Theorem \ref{na-dobro1-ly-cont} holds for
$((e^{-\epsilon_{j}t}T_{\varphi_{j}}(t))_{t\geq 0})_{1\leq j\leq N},$
with $X_{0}:={\mathcal D}({\mathbb R}^{m}).$ The condition (ii) of Theorem \ref{na-dobro1-ly-cont} also holds and, in this concrete situation, we have that $\lim_{t\rightarrow \infty}\|e^{-\epsilon_{j}t}T_{\varphi_{j}}(t)y\|=+\infty ,$ $j\in {\mathbb N}_{N}$
with the vector $y:=|\cdot|.$ Therefore, the strongly continuous semigroups $((e^{-\epsilon_{j}t}T_{\varphi_{j}}(t))_{t\geq 0})_{1\leq j\leq N}$ are densely $(d,f)$-distributionally chaotic for each function $f \in {\mathrm F}.$ 
\end{example}

\end{document}